\documentclass[11pt]{article}
\pdfoutput=1
\pdfoutput=1
\usepackage{fullpage}
\usepackage{fancyhdr}
\usepackage{epsfig}
\usepackage{amsmath,amssymb,amsthm}
\usepackage{dsfont}
\usepackage{mathtools}
\usepackage{tikz}
\usepackage{wrapfig}
\usepackage{graphicx}
\usepackage{tikz-cd}
\usepackage{bbm}
\usepackage[export]{adjustbox}
\usepackage[cmtip,all]{xy}
\usepackage{caption}

\makeatletter
\DeclareRobustCommand{\cev}[1]{%
  \mathpalette\do@cev{#1}%
}
\newcommand{\do@cev}[2]{%
  \fix@cev{#1}{+}%
  \reflectbox{$\m@th#1\vec{\reflectbox{$\fix@cev{#1}{-}\m@th#1#2\fix@cev{#1}{+}$}}$}%
  \fix@cev{#1}{-}%
}
\newcommand{\fix@cev}[2]{%
  \ifx#1\displaystyle
    \mkern#23mu
  \else
    \ifx#1\textstyle
      \mkern#23mu
    \else
      \ifx#1\scriptstyle
        \mkern#22mu
      \else
        \mkern#22mu
      \fi
    \fi
  \fi
}

\makeatother

\makeatletter
\def\slashedarrowfill@#1#2#3#4#5{%
  $\m@th\thickmuskip0mu\medmuskip\thickmuskip\thinmuskip\thickmuskip
\relax#5#1\mkern-7mu%
\cleaders\hbox{$#5\mkern-2mu#2\mkern-2mu$}\hfill
\mathclap{#3}\mathclap{#2}%
\cleaders\hbox{$#5\mkern-2mu#2\mkern-2mu$}\hfill
\mkern-7mu#4$%
}
\def\rightslashedarrowfill@{%
  \slashedarrowfill@\relbar\relbar\mapstochar\rightarrow}
\newcommand\xslashedrightarrow[2][]{%
  \ext@arrow 0055{\rightslashedarrowfill@}{#1}{#2}}
\makeatother

\usepackage[numbers]{natbib}
\theoremstyle{definition}
\newtheorem{definition}{Definition}
\numberwithin{definition}{subsection}

\newtheorem{lemma}{Lemma}
\numberwithin{lemma}{subsection}
\newtheorem{construction}{Construction}
\numberwithin{construction}{subsection}

\numberwithin{notation}{subsection}

\numberwithin{observation}{subsection}

\numberwithin{conjecture}{subsection}
\newtheorem{theorem}{Theorem}
\numberwithin{theorem}{subsection}

\numberwithin{corollary}{subsection}
\newtheorem{proposition}{Proposition}
\numberwithin{proposition}{subsection}
\newtheorem*{acknowledgements}{Acknowledgements}

\DeclareMathOperator{\Da}{\mathcal{D}}
\DeclareMathOperator{\Ea}{\mathcal{E}}

\DeclareMathOperator{\Ra}{\mathcal{R}}

\DeclareMathOperator{\Sa}{\mathcal{S}}
\DeclareMathOperator{\Ta}{\mathcal{T}}
\DeclareMathOperator{\Ka}{\mathcal{K}}

\DeclareMathOperator{\R}{\mathbb{R}}

\DeclareMathOperator{\id}{\textbf{id}}
\DeclareMathOperator{\inv}{^{-1}}

\DeclareMathOperator{\im}{\text{im}}

\DeclareMathOperator{\interior}{\text{int}}
\DeclareMathOperator{\inV}{\text{in}_v}

\DeclareMathOperator{\outV}{\text{out}_v}

\DeclareMathOperator{\inE}{\text{in}}
\DeclareMathOperator{\outE}{\text{out}}
\DeclareMathOperator{\topro}{\xslashedrightarrow{}} 

\newcommand{\nanograph}[1]{\begin{tabular}{@{}c@{}}\includegraphics[scale=.125]{#1}\end{tabular}}
\newcommand{\tinygraph}[1]{\begin{tabular}{@{}c@{}}\includegraphics[scale=.25]{#1}\end{tabular}}
\newcommand{\midgraph}[1]{\begin{tabular}{@{}c@{}}\includegraphics[scale=.40]{#1}\end{tabular}}
\newcommand{\ingraph}[1]{\begin{tabular}{@{}c@{}}\includegraphics[scale=.5]{#1}\end{tabular}}
\newcommand{\fullgraph}[1]{\begin{tabular}{@{}c@{}}\includegraphics{#1}\end{tabular}}

\newcommand{\longsquiggly}{\xymatrix{{}\ar@{~>}[r]&{}}}
\newcommand{\xto}[1]{\xrightarrow{#1}}

\newcommand{\comment}[1]{}
\newcommand{\op}{^\text{op}}

\newbox\gnBoxA
\newdimen\gnCornerHgt
\setbox\gnBoxA=\hbox{$\ulcorner$}
\global\gnCornerHgt=\ht\gnBoxA
\newdimen\gnArgHgt
\def\godelnum #1{%
\setbox\gnBoxA=\hbox{$#1$}%
\gnArgHgt=\ht\gnBoxA%
\ifnum     \gnArgHgt<\gnCornerHgt \gnArgHgt=0pt%
\else \advance \gnArgHgt by -\gnCornerHgt%
\fi \raise\gnArgHgt\hbox{$\ulcorner$} \box\gnBoxA %
\raise\gnArgHgt\hbox{$\urcorner$}}

\title{String Diagrams for Double Categories and Equipments}
\date{}
\author{David Jaz Myers}

\begin{document}

\graphicspath{ {graphics/} }


\maketitle

\begin{abstract}
    A popular graphical calculus for monoidal categories makes computations tactile and intuitive. Complicated diagram chases can be expressed in a few pictures and discovered by playing with a shoelace. Joyal and Street's proof of the soundness of this calculus says that any deformation of a diagram, any bending of the strings, describes the same morphism. In this paper, we extend the graphical calculus to double categories and proarrow equipments in order to bring the string diagrammatic method to formal category theory. Our main theorem proves this calculus sound with the help of Dawson and Par{\'e}'s results on composition in double categories.
\end{abstract}

\section*{Introduction}

String diagrams provide a graphical calculus used most often to describe monoidal (or tensor) categories and their variants. Examples include flow charts, Feynman diagrams, circuit diagrams, Petri nets, Markov processes, and knot diagrams. Because string diagrams are not written sequentially like traditional notation, they are good for describing parallel processes as they arise in a variety of situations. Even better, any suitable deformation of a string diagram describes the same morphism in its respective monoidal category. This lets us compute a complicated composite with simple manipulations of string. In this paper, we extend string diagram notation to double categories and proarrow equipments. Equipments were first introduced by Wood in \cite{Wood}, but are considered here in the slightly more general and standard form known also as a \emph{framed bicategory}, for example in \cite{Shulman2007}.

If categories are the abstract algebras of functions, then \textit{equipments} are the abstract algebras of functions and relations. Where categories consist of objects and arrows between them, equipments consist of four kinds of things: objects; vertical arrows, which are meant to behave like functions; horizontal arrows, which are meant to behave like relations; and 2-cells, which act like implications between relations. The following are a few common examples of equipments:
\begin{itemize}
    \item The equipment of sets, functions, and relations, whose objects are sets, vertical arrows are functions, horizontal arrows are relations, and 2-cells are implications.
    \item The equipment of rings, homomorphisms, and bimodules, whose 2-cells are bimodule morphisms.
    \item The equipment of categories, functors, and profunctors, whose 2-cells are profunctor morphisms.
    \item Generalizing all the above, the equipment of categories enriched in $\mathcal{V}$, $\mathcal{V}$-functors, and $\mathcal{V}$-profunctors, whose 2-cells are morphisms of $\mathcal{V}$-profunctors.
\end{itemize}

A good amount of formal category theory can be carried out in an equipment, specializing to the expected concepts in the equipments of enriched categories. Wood defined equipments in \cite{Wood} for this purpose, and they continue to be used to study enrichment and its generalizations, e.g.~ by Shulman \cite{ShulmanEIC}. 

Equipments are, in particular, a kind of double category. A double category has objects, vertical arrows, horizontal arrows, and 2-cells supporting two sorts of composition: horizontal and vertical. 
\setlength{\intextsep}{2pt}%
\begin{wrapfigure}{r}{0.25\textwidth} 
    \centering
    \includegraphics[width=0.25\textwidth]{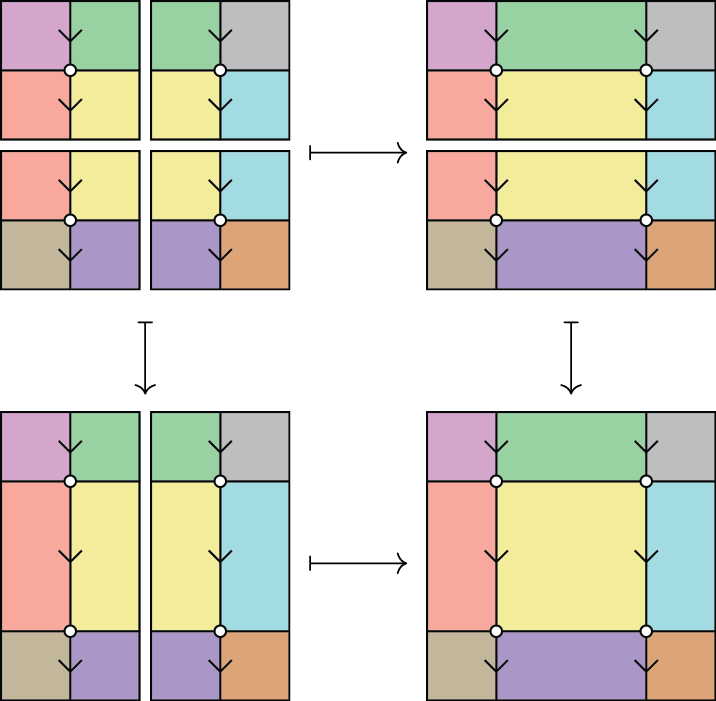}
    \captionsetup{labelformat=empty}
    \caption{\footnotesize Horizontal and vertical composition in the graphical calculus.}
\end{wrapfigure} 
Double categories are used in the study of universal 2-algebra (see e.g. \cite{Kelly2} and \cite{Fiore}), where the vertical and horizontal arrows represent lax and colax morphisms. Double categories have also been used by \cite{Bruni} in the study of rewriting, where the extra direction of arrows takes into account effects and synchronization of rewrites. The mate calculus in double categories was used, for example, by Shulman to study composites of left and right derived functors in \cite{Shulman2011}.

The string diagram graphical calculus for double categories and equipments presented here is the Poincar{\'e} dual of the usual square notation: objects are regions, vertical morphisms are vertical strings, horizontal morphisms are horizontal strings, and squares become beads on these strings. In equipments, vertical strings may be bent horizontally to express the algebra of \textit{companions} and \textit{conjoints}.
\setlength{\intextsep}{2pt}%
\begin{wrapfigure}{l}{0.25\textwidth} 
    \centering
    $\ingraph{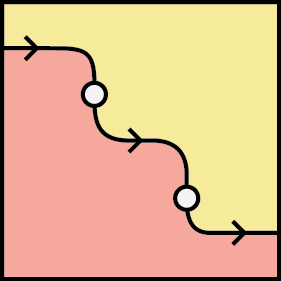} = \ingraph{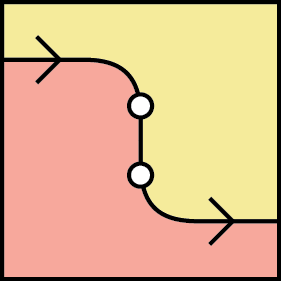}$
\end{wrapfigure} 
In double categories of 2-algebras, these same bends represent pseudo-morphisms (which are both lax and colax) and doctrinal adjunctions respectively. The graphical calculus presented here is also applicable to any vertical arrow in a double category which happens to have a companion or a conjoint, and is therefore applicable also to double categories used in 2-algebra. It is the author's hope that this notation will make working with double categories and equipments easier and more intuitive.

This paper is made up of three parts. In Part 1, we introduce the graphical calculus, and use it to prove a few basic lemmas about equipments. In particular, we discuss the graphical interpretation of the mate calculus in an equipment, and prove the equivalence of the natural transformation and hom-set definitions of an adjunction with the string diagrams. These elementary examples show how the calculus can be used in practice.

In Parts 2 and 3, we prove that the interpretation of a string diagram is invariant under deformation. Part 2 concerns diagrams for double categories. In \cite{Joyal} and \cite{Joyal2}, Joyal and Street define embedded graphs with boundary meant to represent morphisms in various flavors of monoidal category, and prove that the interpretation of these diagrams is invariant under deformation. We will follow their general program and define double and equipment diagrams in a similar manner.

Not all arrangements of 2-cells expressible in the language of double categories admit a composite under the two compositions. Dawson and Par{\'e} undertook a general study of composition in double categories in \cite{Dawson1993} and \cite{Pare} using the notions of rectangle tilings and tile orders. 
\setlength{\intextsep}{2pt}%
\begin{wrapfigure}{l}{0.25\textwidth} 
    \centering
    \includegraphics[width=0.25\textwidth]{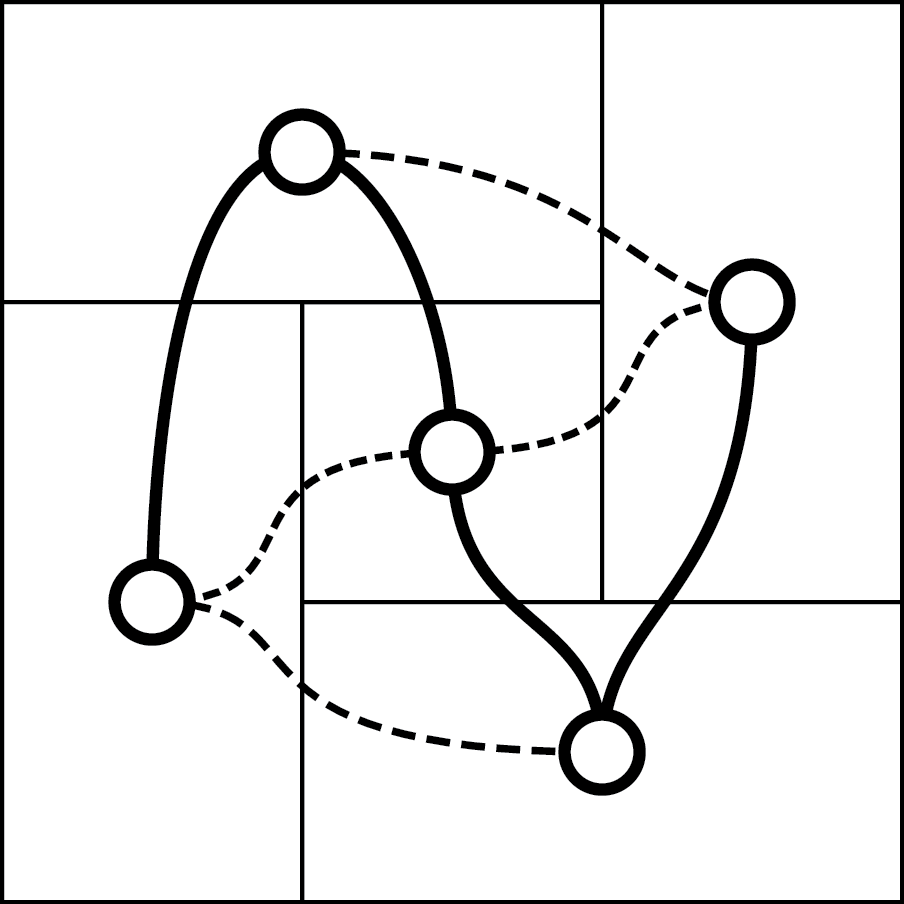}
    \captionsetup{labelformat=empty}
    \caption{\footnotesize The pinwheel (in this and reverse orientation) is the only obstruction to composition in a double category.}
\end{wrapfigure}
We use their tiling results to give conditions under which double diagrams admit composites, and to show that double and equipment diagrams admit unique composites. In \cite{Dawson}, Dawson characterizes the possible obstructions to composition; we use this result to show that double diagrams satisfying a natural condition called \emph{neatness} (after Dawson's similar notion for tile orders) admit composites by avoiding these obstructions. 

Part 3 concerns diagrams for equipments. Since a deformation of equipment diagrams may change the signature --- the boundary data --- of 2-cells, we briefly study what sort of changes may occur and the kinds of effects on the value of a diagram which can result from such a change. We then show that the value of an equipment diagram is invariant under deformation so long as its interpretation changes according to the changes in the signatures of the 2-cells inside it. This establishes the soundness of the graphical mate calculus for double categories and equipments. 

\begin{acknowledgements}
I would like to thank John Baez for suggesting that I take my notes on these diagrams and prove that they are invariant under deformation, and Jack Calcut for discussing the topological aspects of double diagrams. I would like to thank Emily Riehl and Mike Shulman for reading through the drafts of this paper; their thorough and insightful comments helped me a great deal.
\end{acknowledgements}

\tableofcontents

\section{String Diagrams}

\subsection{Definitions}
Let's recall the definition of a double category and introduce the new notation. A double category is a category internal to the category of categories, and so consists of the following data:
\begin{definition}
A \emph{double category} $\Ea$ has
\begin{enumerate}
\item Objects $A$, $B$, $C,$ $\ldots$, which will be written as bounded plane regions of different colors $\ingraph{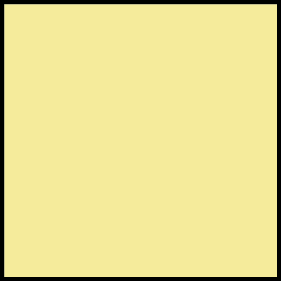}$, $\ingraph{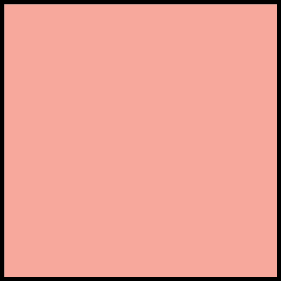}$, $\ingraph{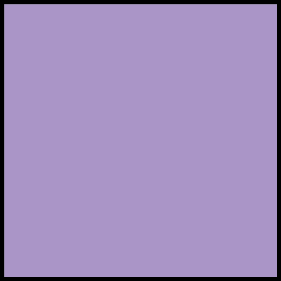}$, $\ldots$. 
\item Vertical arrows $f$, $g$, $h : \tinygraph{oR} \rightarrow \tinygraph{oY}$, $\ldots$, which we will just call \emph{arrows} and write as vertical lines $\ingraph{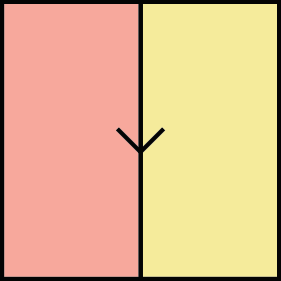}$, directed downwards, dividing the plane region $\ingraph{oR}$ from $\ingraph{oY}$.
\item Horizontal arrows $J$, $K$, $H : \tinygraph{oR} \topro \tinygraph{oY}$, $\ldots$, which we will just call \emph{proarrows} and write as horizontal lines $\ingraph{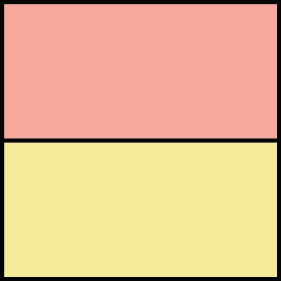}$ dividing the plane region $\ingraph{oR}$ from $\ingraph{oY}$.
\item 2-cells $\alpha$, $\beta$, $\ldots$, which we write as beads at the intersection of vertical and horizontal lines.
\[\fullgraph{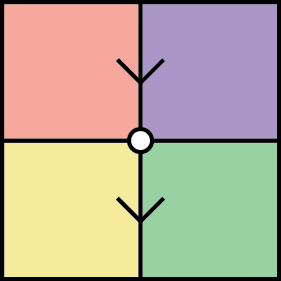} \quad\quad\quad\quad \begin{tikzcd}
A \arrow{rr}{J} \arrow{dd}[swap]{f} & & B \arrow{dd}{g}\\
& \alpha & \\
C \arrow{rr}[swap]{K}  & & D
\end{tikzcd}\]

\end{enumerate}

Note that the string notation is the Poincar\'{e} dual of the transpose of the usual arrow notation (on the right above). We transpose so vertical arrows will be drawn with vertical lines, and horizontal arrows with horizontal lines, and so that we can easily remember facts about companions and conjoints (which are introduced below) while reading 2-cells top to bottom, left to right in the case of enriched categories. From now on we will also transpose the usual double category notation, so that the two are simply dual.

\end{definition}

We write composition of arrows and proarrows by juxtaposition, and composition of 2-cells by joining matching lines:
$$\ingraph{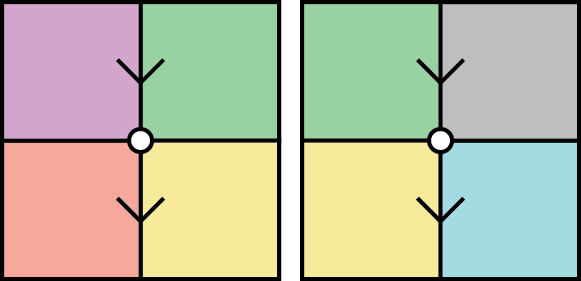} \mapsto \ingraph{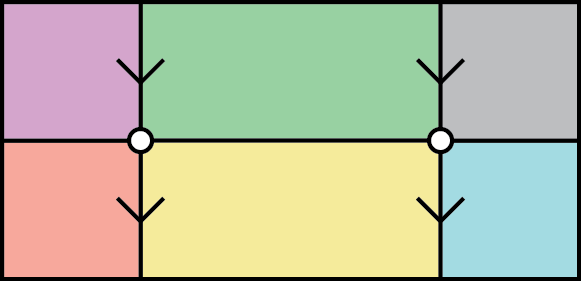}.$$
This is horizontal composition, and in more traditional notation would be written $\alpha,\, \beta \mapsto \alpha \mid \beta$. Vertical composition, where the vertical arrows are joined and the horizontal arrows composed, would be written $\alpha,\,\beta \mapsto \frac{\alpha}{\beta}$.

We read composition of arrows left to right, composition of proarrows top to bottom, and 2-cells top-bottom and left-right. We represent the identity 2-cell on an object $\tinygraph{oR}$ by the same same region of colors; similarly, the identity 2-cells on arrows $\tinygraph{vRY}$ and proarrows $\tinygraph{hRY}$ are represented by just these diagrams. In this way, the identity laws become graphically obvious. 

Composition satisfies the interchange law, which says that composing horizontally and then vertically is the same as composing vertically and then horizontally.
\[\fullgraph{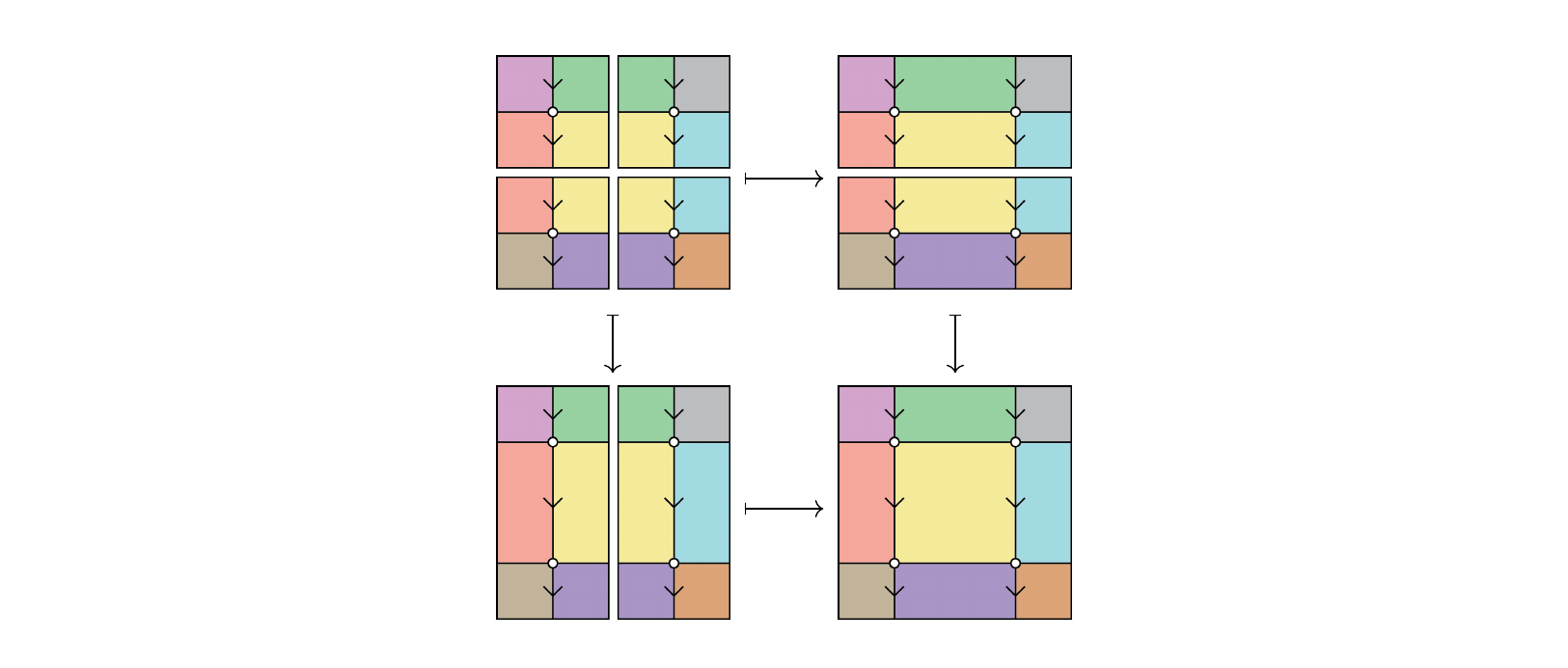}\]

In terms of diagrams, the interchange law says that the diagram in the bottom right corner above has a well-defined interpretation. In terms of the a traditional notation, it says the following equality holds for composable 2-cells:
$$\left. \frac{\alpha}{\beta} \middle| \frac{\gamma}{\delta} \right. = \frac{\alpha \mid \gamma}{\beta \mid \delta}.$$

\subsubsection{Note (a Word on Choices of Duals)}

There are 8 ways we could interpret our diagrams in a double category, corresponding to the symmetries of the square. These options are cut into 4 by choosing to mark the ``arrow'' direction as vertical, and the ``proarrow'' direction as horizontal when interpreting in an equipment of enriched categories. We will use this convention -- which is used, for example, in \cite{Leinster} and \cite{Cruttwell} -- throughout the rest of the paper, but it is not universal. 

We take the view that composition arrows and proarrows should be read top to bottom and left to right. However, depending on the definitions of certain equipments, this could force horizontal composition of 2-cells to be read right to left. The standard example of an equipment of enriched categories exhibits this. If we define define a profunctor $A \topro B$ to be $A\op \times B \to V$, then with our conventions the horizontal composite of 2-cells in this equipment may be read left to right; but if a profunctor $A \topro B$ is defined to be $A \times B\op \to V$ (so that it corresponds to a Kleisli-like morphism $A \to V^{B\op}$), then horizontal composition of 2-cells must be read the other way around. This, in particular, changes what we call a companion into a conjoint and vice-versa. We will always imagine a profunctor being defined as $A\op \times B \to V$ in this paper.

Ultimately, these are minor issues. But care must be taken to make sure that all conventions align. 

\subsection{Companions, Conjoints, and the Spider Lemma}\label{Sec_SpiderLemma}

The notions of companion and conjoint are of central importance in the theory of double categories. Importantly, a double category for which every arrow has a companion and a conjoint is a proarrow equipment, which proves a useful setting for doing formal category theory. In this section, we'll introduce notation for companions and conjoints that makes their use very intuitive.

\begin{definition}
An arrow $\tinygraph{vRY}$ has a \emph{companion} if there is a proarrow $\tinygraph{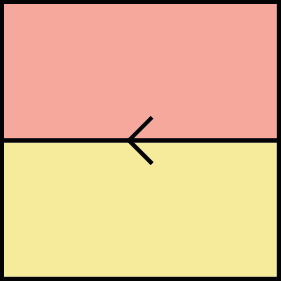}$ together with two 2-cells $\tinygraph{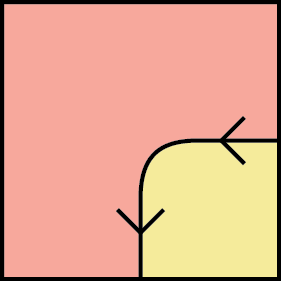}$ and $\tinygraph{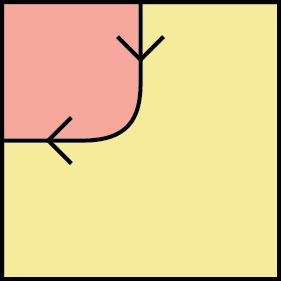}$ such that
\[\ingraph{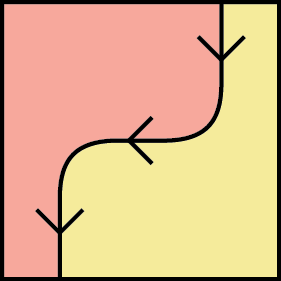} = \ingraph{vRY} \quad\quad\text{and}\quad\quad
\ingraph{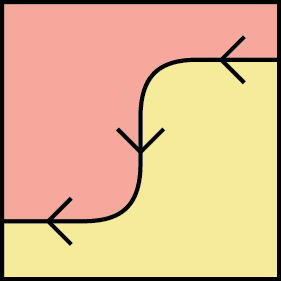} = \ingraph{hRYl}.
\]

Similarly, $\tinygraph{vRY}$ is said to have a \emph{conjoint} if there is a proarrow $\tinygraph{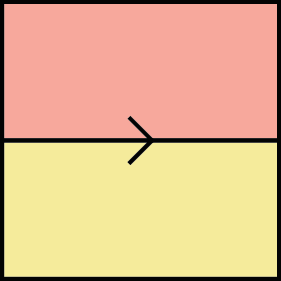}$ together with two 2-cells $\tinygraph{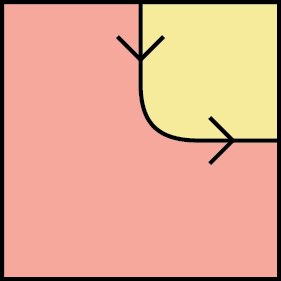}$ and $\tinygraph{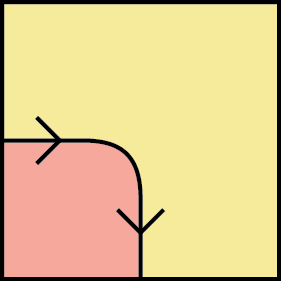}$ such that
\[
\ingraph{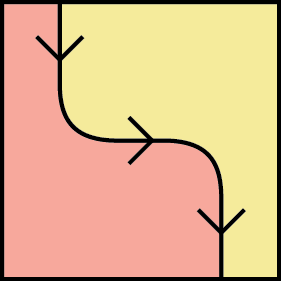}  = \ingraph{vRY} \quad\quad\text{and}\quad\quad
\ingraph{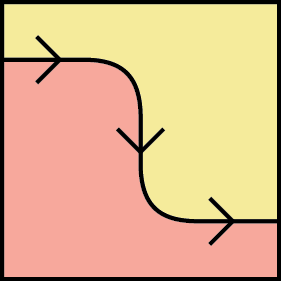} = \ingraph{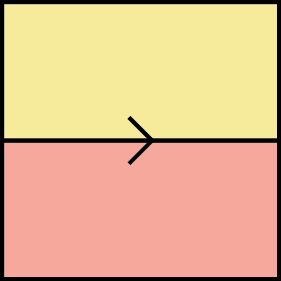}.
\]
\end{definition}

Let's unpack these definitions now. We draw the companion $\nanograph{oY}(f, 1) : \tinygraph{oR} \topro \tinygraph{oY}$ of an arrow $f : \ingraph{vRY}$, if it exists, as a directed horizontal line:
\begin{center}
\ingraph{hRYl}
\end{center}

In arrow notation (transposed), the defining 2-cells (called the unit and counit respectively) of a companion are
\[ \begin{tikzcd}
\tinygraph{oR} \arrow{rr}{\id} \arrow{dd}[swap]{\id} & & \tinygraph{oR} \arrow{dd}{\nanograph{oY}(f,1)}\\
& \beta & \\
\tinygraph{oR} \arrow{rr}[swap]{f}  & & \tinygraph{oY}
\end{tikzcd} \quad\quad\text{and}\quad\quad  \begin{tikzcd}
\tinygraph{oR} \arrow{rr}{f} \arrow{dd}[swap]{\nanograph{oY}(f,1)} & & \tinygraph{oY} \arrow{dd}{\id}\\
& \alpha & \\
\tinygraph{oY} \arrow{rr}[swap]{\id}  & & \tinygraph{oY}
\end{tikzcd}\] 

These satisfy the equalities $\frac{\beta}{\alpha} = \id$ and $\beta \mid \alpha = \id$. Dualizing, we see that $\alpha$ and $\beta$ should be beads forming a corner between $f$ and $\nanograph{oY}(f,1)$. We will suppress the names $\alpha$ and $\beta$ and instead write them as smooth bends:
$$  \ingraph{Bend_BR_RY}\quad\quad\text{and}\quad\quad \ingraph{Bend_LT_RY},$$

We can remember that the companion $\nanograph{oY}(f,1)$ of $f$ is a line pointing to the left, since $f$ appears on the left in $\nanograph{oY}(f,1)$. The equations $\frac{\beta}{\alpha} = \id$ and $\beta \mid \alpha = \id$ then say that we can pull kinks straight.

\begin{align*}
\fullgraph{ZigZagLv_RY} &= \fullgraph{vRY} \\
\fullgraph{ZigZagLh_RY} &= \fullgraph{hRYl}
\end{align*}

In an equipment of enriched categories, an arrow $f$ is a functor, and its companion $\nanograph{oY}(f, 1)$ is the profunctor with components given by the hom $\nanograph{oY}(fa, b)$.

Similarly, we draw the conjoint  $\nanograph{oY}(1, f) : \tinygraph{oY} \topro \tinygraph{oR}$ of $f : \ingraph{vRY}$ as a directed horizontal line:
\begin{center}
\ingraph{hYRr}
\end{center}

We draw the line flowing to the right because $f$ appears on the right in $\nanograph{oY}(1, f)$. The conjoint has its own attendant 2 cells, 

$$ \ingraph{Bend_RT_RY}\quad\quad\text{and}\quad\quad \ingraph{Bend_BL_RY},$$

satisfying their own kink-pulling identities:
\begin{align*}
\fullgraph{ZigZagRv_RY}  &= \fullgraph{vRY}\\
\fullgraph{ZigZagRh_RY} &= \fullgraph{hYRr}
\end{align*}

Companions and conjoints are determined uniquely up to unique isomorphism by the data which describes them. This justifies calling $\ingraph{hRYl}$ \emph{the} companion and $\ingraph{hYRr}$ \emph{the} conjoint of $\ingraph{vRY}$. 

\begin{lemma}
The companion of an arrow is unique up to a unique isomorphism which commutes with its bends.
\end{lemma}
\begin{proof}
Suppose an arrow $\ingraph{vRY}$ has two conjoints $\ingraph{hRY}$ and $\ingraph{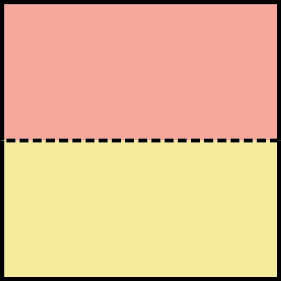}$ with units $\ingraph{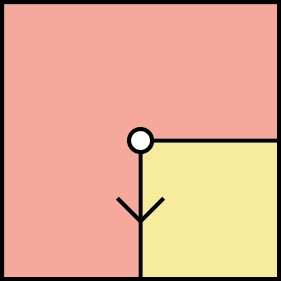}$ and $\ingraph{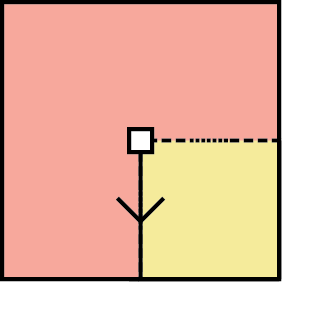}$ and counits $\ingraph{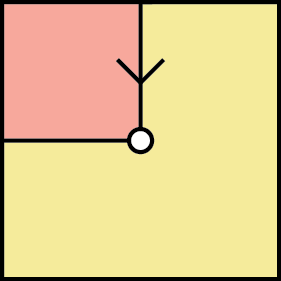}$ and $\ingraph{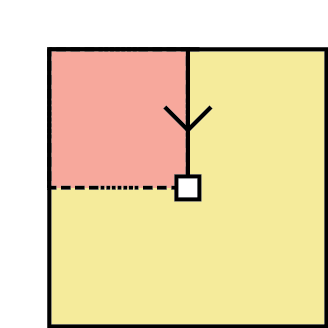}$ respectively. The following calculations show that the maps $\ingraph{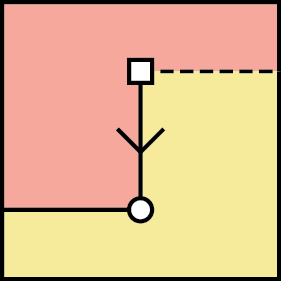}$ and $\ingraph{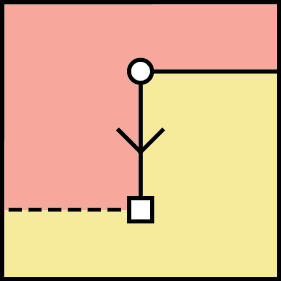}$ form an isomorphism $\ingraph{hRY} \cong \ingraph{hRYdash}$.
\begin{align*}
    \fullgraph{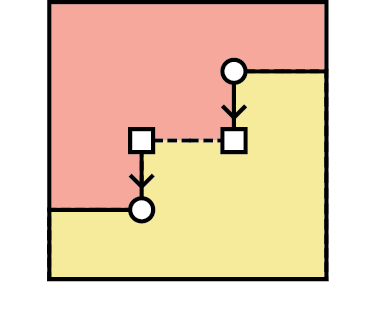} &= \fullgraph{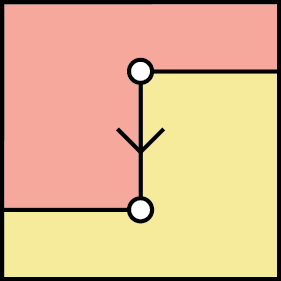} = \fullgraph{hRY}, \\
    \fullgraph{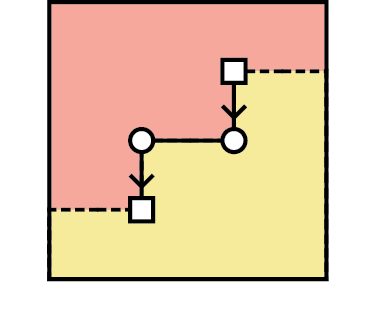} &= \fullgraph{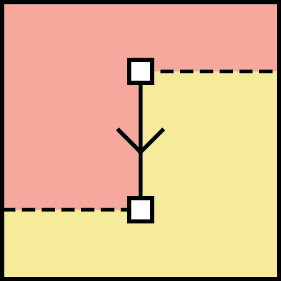} = \fullgraph{hRYdash}.
\end{align*}

An isomorphism $\ingraph{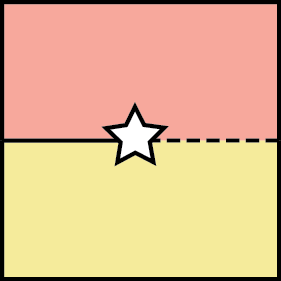}$ with inverse $\ingraph{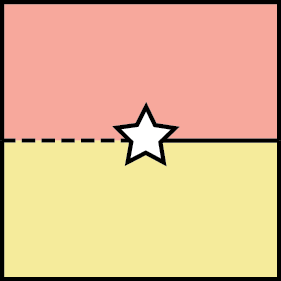}$ is said to commute with the bends of these companions if $\ingraph{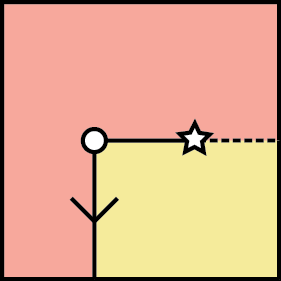} = \ingraph{uniqUnitdash}$ and all the other such equations hold (and similarly for its inverse). The isomorphism $\ingraph{uniqIso}$ commutes with the bends by the kink identities: $\ingraph{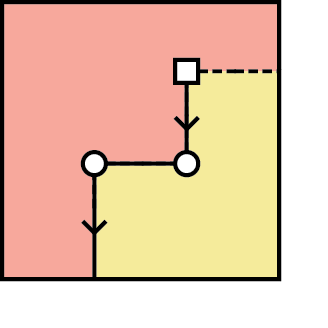} = \ingraph{uniqUnitdash}$ etc. 

Suppose that $\ingraph{hRYstar}$ were an isomorphism which commutes with the bends, then 
$$\fullgraph{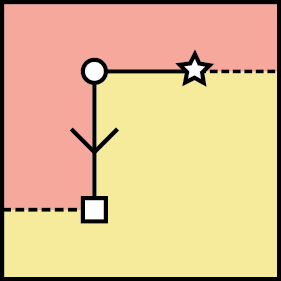} = \fullgraph{uniqZigZagdash} = \fullgraph{hRYdash},$$
so that $\ingraph{hRYstar}$ equals $\ingraph{uniqIso}$. Therefore, $\ingraph{uniqIso}$ is unique.
\end{proof}

The following lemma is a standard elementary result about companions and conjoints. It is easily proved with the string diagrams.

\begin{lemma}
For an arrow $\ingraph{vRY}$ and proarrows $\ingraph{hRY}$ and $\ingraph{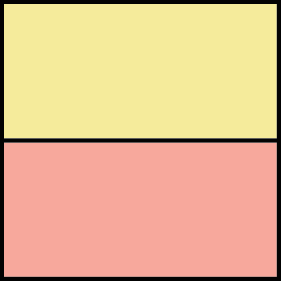}$, any two of the following implies the third:
\begin{enumerate}
\item $\tinygraph{hRY}$ is the companion of $\tinygraph{vRY}$,
\item $\tinygraph{hYR}$ is the conjoint of $\tinygraph{vRY}$,
\item $\tinygraph{hRY}$ is the left adjoint of $\tinygraph{hYR}$ in the category of proarrows.
\end{enumerate}
\end{lemma}
\begin{proof}
First let's show that 1 and 2 imply 3. We define the unit and counit by bending all the way around:
$$\ingraph{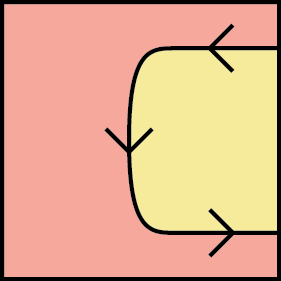} \quad\quad\text{and}\quad\quad \ingraph{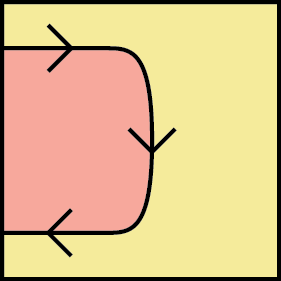}.$$

The zig-zag identities (also known as the triangle identities) then follow simply by pulling the strings straight:
$$\fullgraph{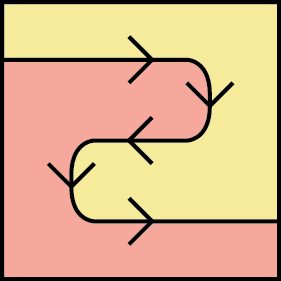} = \fullgraph{ZigZagRh_RY} = \fullgraph{hYRr}.$$

The other side goes similarly: $\ingraph{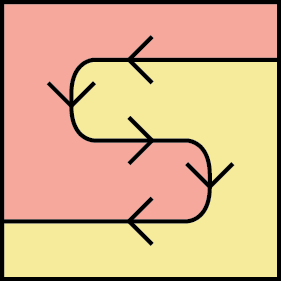} = \ingraph{ZigZagLh_RY} = \ingraph{hRYl}$. 

Now, suppose 2 and 3, and we'll show that $\tinygraph{hRY}$ is the companion of $\tinygraph{vRY}$. Let's write the unit and counit of the adjunction $\tinygraph{hRY} \dashv \tinygraph{hYRr}$ as $\ingraph{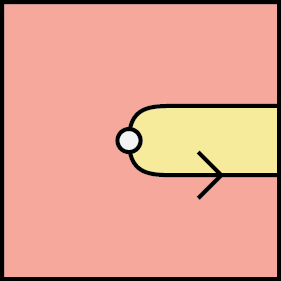}$ and $\ingraph{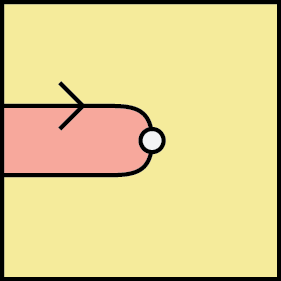}$. Then  $\ingraph{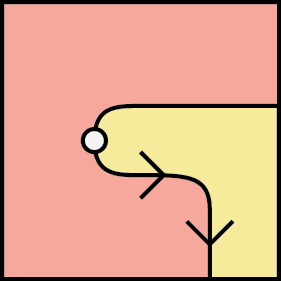}$ and $\ingraph{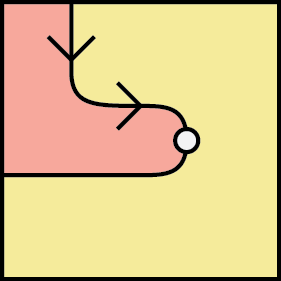}$ are the cells making $\tinygraph{hRY}$ into the companion of $\tinygraph{vRY}$. These satisfy the kink identities thanks to the kinks of $\tinygraph{hYRr}$ and the zig-zags of $\tinygraph{hRY} \dashv \tinygraph{hYRr}$:
$$\fullgraph{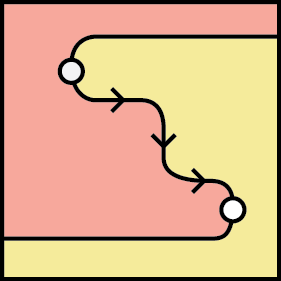} = \fullgraph{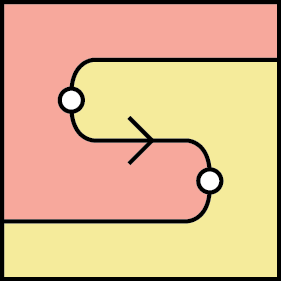} = \fullgraph{hRY},$$
$$\fullgraph{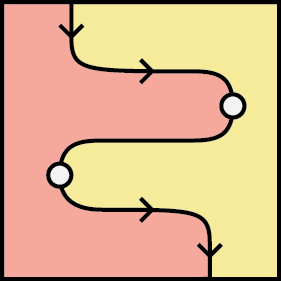} = \fullgraph{ZigZagRv_RY} = \fullgraph{vRY}.$$

That 1 and 3 imply 2 follows similarly.
\end{proof}

\begin{definition}
A \emph{proarrow equipment} is a double category where every arrow has a conjoint and a companion.
\end{definition}

The following lemma is a central elementary result of the theory of equipments:

\begin{lemma}[Spider Lemma]\label{Lem_SpiderLemma}
In an equipment, we can bend arrows. More formally, there is a bijective correspondence between diagrams of form of the left, and diagrams of the form of the right:
$$\fullgraph{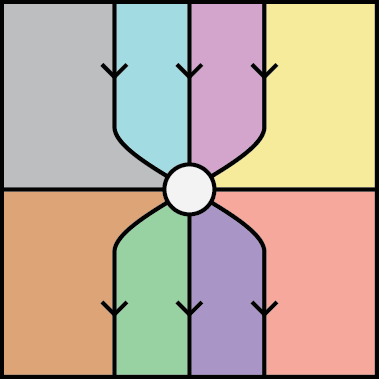} \simeq \fullgraph{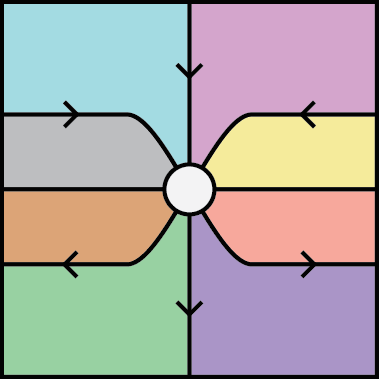}.$$
\end{lemma}
\begin{proof}
The correspondence is given by composing the outermost vertical or horizontal arrows by their companion or conjoint (co)units, as suggested by the slight bends in the arrows above. The kink identities then ensure that these two processes are inverse to each other, giving the desired bijection.
\end{proof}

The Spider Lemma is called the ``Central Lemma'' on the $n$Lab article \cite{nLab:Equipment}. With this lemma in hand, we can begin to really use the graphical notation. Whenever we invoke the spider lemma, we will simply say that we are ``bending arrows''. Even better, the assignment of arrows to their companions (or conjoints) is functorial, which means that we can ``slide beads around bends''.

Given $\ingraph{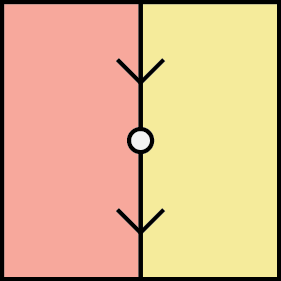}$, define its \emph{conjoint mate} $\ingraph{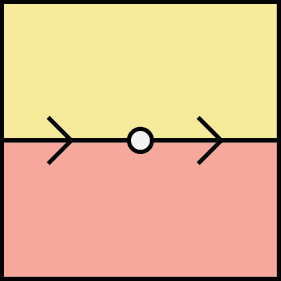}$ to be $\ingraph{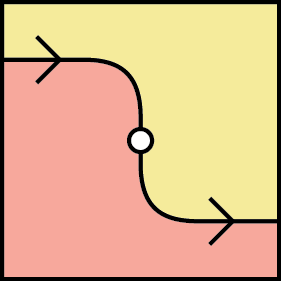}$. The kink identities say that the conjoint mate of the identity is the identity of the conjoint, and $\ingraph{ZigZag2Cell_Funct2} = \ingraph{ZigZag2Cell_Funct1}$. Therefore, the assignment of a bead to its conjoint mate is functorial. Thanks to the spider lemma, this functor is fully faithful.

We can now slide beads around bends thanks to the following equations:
$$\fullgraph{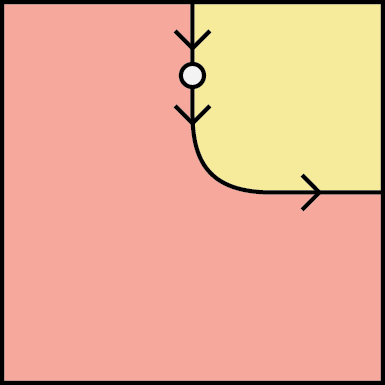} = \fullgraph{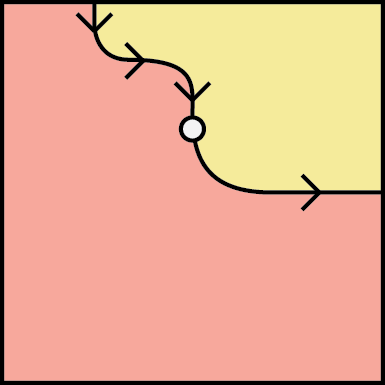} = \fullgraph{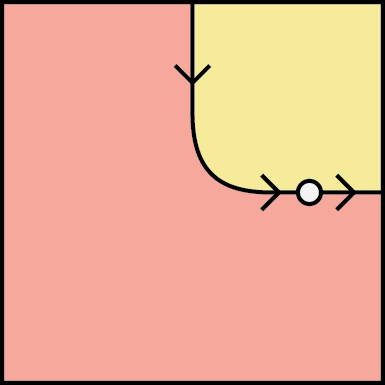}.$$
In fact, any topological deformation of a diagram that preserves the fact that it is a diagram induces an equality in the equipment being described. We will prove this in Sections \ref{Sec_DoubleDiagrams} and \ref{Sec_EquipmentDiagrams}.

\subsection{Zig-Zag Adjunctions and Hom-set Adjunctions}
It is a classical fact of category theory that an adjunction $f \dashv g : A \rightleftarrows B$ may be defined 2-categorically or profunctorially. That is to say, we could use natural transformations $\eta: \id \rightarrow gf$ and $\epsilon: fg \rightarrow \id$ (which we will call a zig-zag adjunction, after the coherence conditions), or a natural isomorphism $\psi : B(f, 1) \cong A(1, g)$. This equivalence holds in any proarrow equipment, which we can now show quickly and intuitively with string diagrams.

Suppose we have an adjunction $\ingraph{vRY} \dashv \ingraph{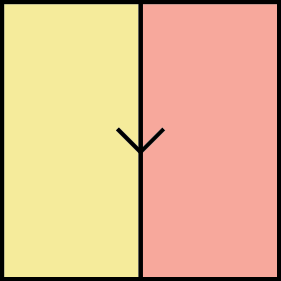}$, given by the vertical cells $\ingraph{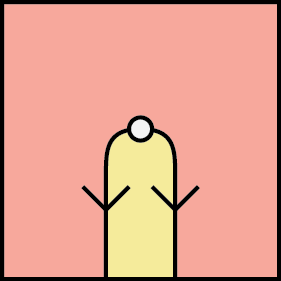}$ and $\ingraph{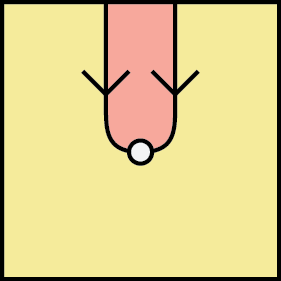}$, satisfying the zig-zag (or, triangle) identities
$$\ingraph{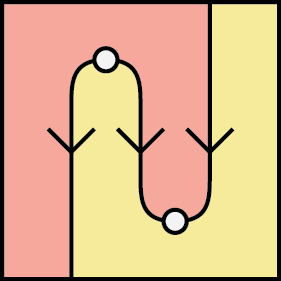} = \ingraph{vRY} \quad\quad\text{and}\quad\quad  \ingraph{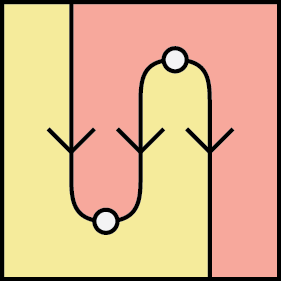} = \ingraph{vYR}.$$

By bending the unit and counit, we get the horizontal cells $\ingraph{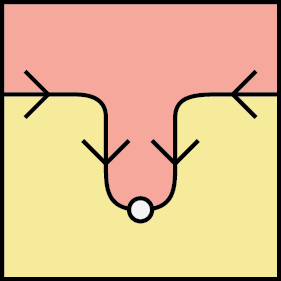}$ and $\ingraph{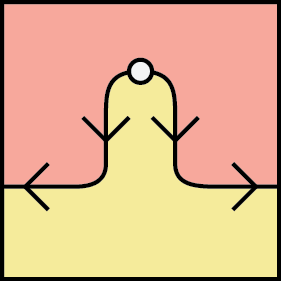}$. Bending the zig-zag identities shows that these maps are horizontally inverse to each other
$$\ingraph{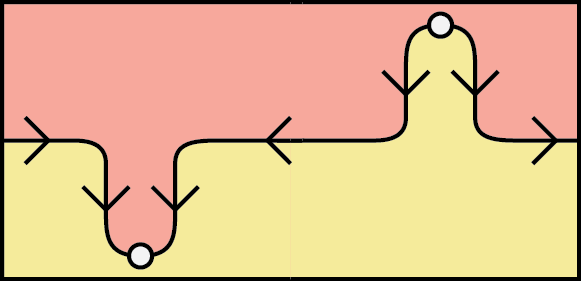} = \ingraph{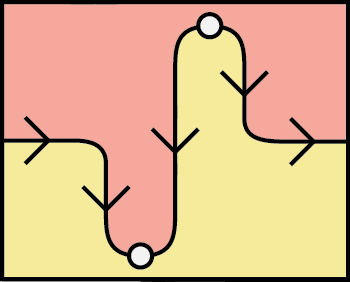} = \ingraph{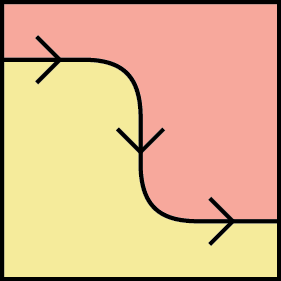} = \ingraph{hRYr},$$
$$\ingraph{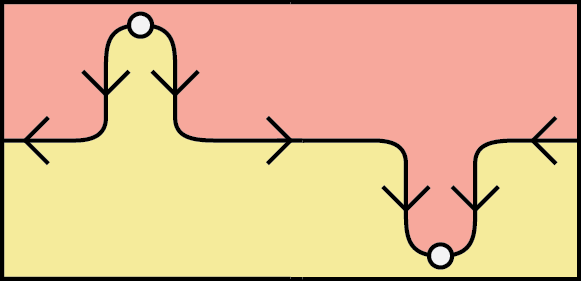} = \ingraph{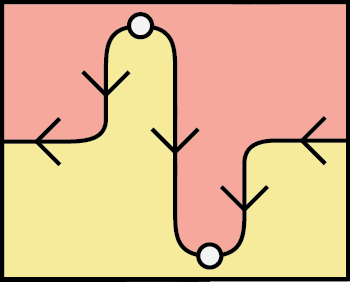} = \ingraph{ZigZagLh_RY} = \ingraph{hRYl},$$
and therefore define the natural isomorphism $\ingraph{hRYr}  \cong \ingraph{hRYl} $ we wanted.

Going the other way, suppose $\ingraph{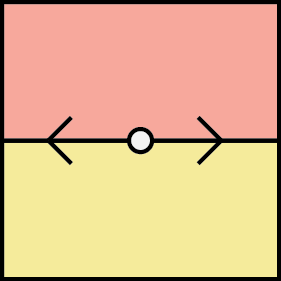}$ is a natural isomorphism with horizontal inverse $\ingraph{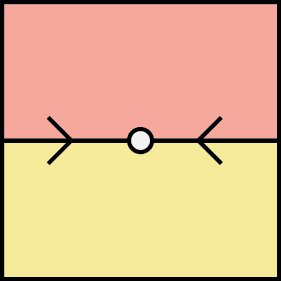}$, meaning
\begin{align*}
    \ingraph{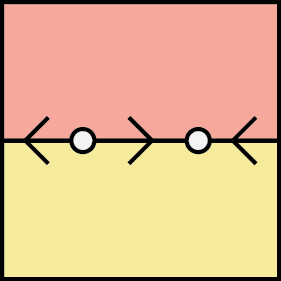} = \ingraph{hRYl} \quad\quad\text{and}\quad\quad  \ingraph{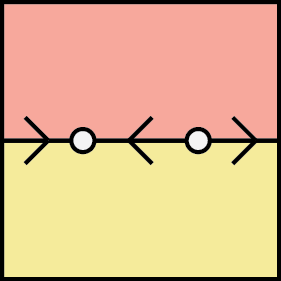} = \ingraph{hRYr}. \tag{1}
\end{align*}

Then we can define a unit $\ingraph{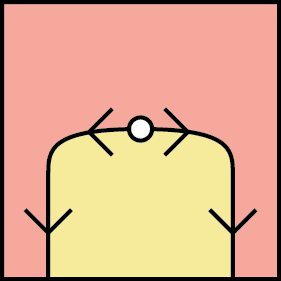}$ and counit $\ingraph{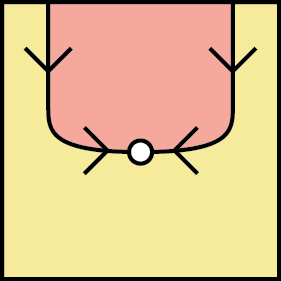}$ by bending. These satisfy the zig-zag identities by pulling straight and using (1):
$$\ingraph{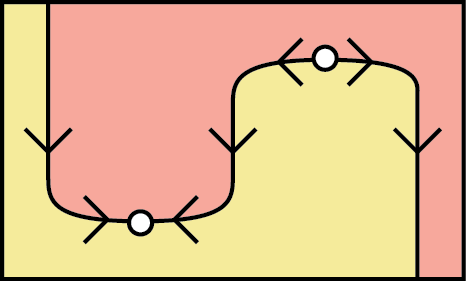} = \ingraph{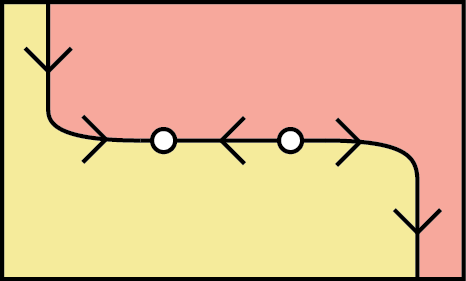} = \ingraph{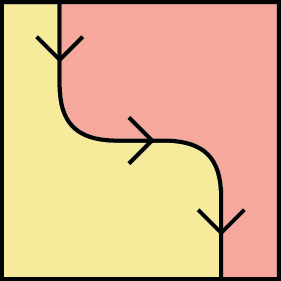} = \ingraph{vYR},$$
$$\ingraph{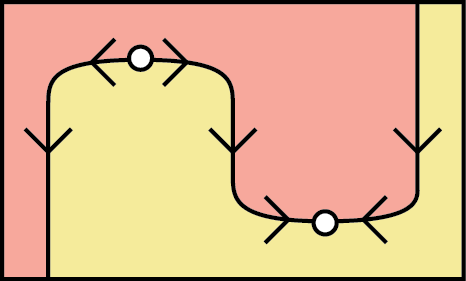} = \ingraph{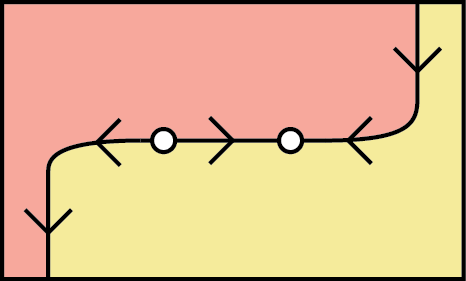} = \ingraph{ZigZagLv_RY} = \ingraph{vRY}.$$

Though this proof can be discovered graphically, it specializes to the usual argument in the case that the equipment is an equipment of enriched categories.

\section{Double Diagrams}\label{Sec_DoubleDiagrams}

In this and the coming sections, we prove the correctness of our diagrammatic language for double categories and equipments. More specifically, we will show that if a diagram describes some composite 2-cell in a double category, then any deformation of that diagram describes the same 2-cell. We will use definitions and methods similar to Joyal and Street in \cite{Joyal}. The plan is to decompose a complicated diagram into simple pieces which may then be associated to a tile-order in the sense of Dawson and Pare (see \cite{Dawson1993} and \cite{Dawson}). By a theorem of Dawson and Pare in \cite{Dawson1993}, the composite of these simple diagrams (if it exists) is unique. Since a suitably small deformation of the original diagram will generate the same tile-order, it will have the same composite. Therefore, the diagrams will have invariant meaning under deformation.

We break this argument into several steps.
\begin{itemize}
    \item In Section 2.1, we recall the preliminary defintions of to Joyal and Street in \cite{Joyal}.
    \item In Section 2.2, we extend those definitions to double diagrams for double categories and define the value of a double diagram.
    \item In Section 2.3, we investigate tilings of double diagrams and prove that the value of a double diagram, if it exists, is invariant under deformation (Lemma \ref{Lem_ValueInvariant}).
    \item In Section 2.4, we show that every double diagram which is \emph{neat} admits a composable tiling, and therefore has a value invariant under deformation (Theorem \ref{Thm:InvariantValue}).
\end{itemize}

\subsection{Graphs and Diagrams}
These definitions follow and extend those of \cite{Joyal}. First, a \emph{graph} $(G, G_0)$ is defined to be a pair of a Hausdorff space $G$ and a finite set of points $G_0$ (called \emph{nodes}) in $G$ such that $G - G_0$ is a one dimensional manifold. Connected components $e \in \pi_0(G - G_0)$ are called \emph{edges}, and we let $G_1 = \pi_0(G - G_0)$ be the set of edges. For a node $n \in G_0$, its degree is defined to be 
$$\deg n := \inf_{n \in U \text{ open}} |\pi_0(U - \{n\})|,$$
the number of connected components within a sufficiently small open neighborhood $U$ of $n$. A \emph{graph with boundary} is a compact graph $(G, G_0)$ together with a subset $\partial G \subseteq G_0$ of degree one nodes. From now on all graphs will have boundary.

We will `draw' our graphs in rectangles in the plane. To prepare for this, let's define some useful notions pertaining to rectangles.
\begin{definition}
Define a \emph{rectangle} to be a subset of the real plane $\R^2$ of the form $[a, b] \times [c, d]$ for $a < b$, $c < d$ real numbers.
\begin{itemize}
    \item If $R$ is a rectangle, define $\partial_v R := [a, b] \times \{c,d\}$ and $\partial_h R := \{a, b\} \times [c, d]$.
    \item Define $vR$ to be the vertical axis $\{a\} \times [c, d]$ and $hR$ to be the horizontal axis $[a, b] \times \{c\}$ and let $p_v : R \rightarrow vR$ and $p_h : R \rightarrow hR$ be the projections respectively.
    \item   For points $x$ and $y$ in a rectangle $R$, we'll say that $x \geq_v y$ (read:\ $x$ is above $y$) if $p_v(x) \geq p_v(y)$, and similarly $x \leq_h y$ ($x$ is left of or before $y$) if $p_h(x) \leq p_h(y)$.
\end{itemize}
\end{definition}
 
 In general, if $X$ is a subspace then we will denote its interior by $\interior X$. Now we are ready to define the notion of a diagram.

\begin{definition} \label{Def_VerticalDiagram}
Let $(G, G_0, \partial G)$ be a graph with boundary, and $R$ a rectangle. A \emph{vertical diagram} of $G$ in $R$ is an embedding $\varphi: G \rightarrow R$, differentiable when restricted to each edge, satisfying the following two properties:
\begin{enumerate}
    \item $\varphi(\partial G) \subseteq \interior \partial_v R$, and if $\varphi(x) \in \partial R$, then $x \in \partial G$.
    \item (Progressivity) For all edges $e \in G_1$, the composite $e \hookrightarrow G \xrightarrow{\varphi} R \xrightarrow{p_v} vR$ is injective. 
\end{enumerate}
\end{definition}

The first condition ensures that boundary nodes of $G$ fall on the boundary of $R$, and that $G$ only intersects the boundary on these nodes. The second condition is called \emph{progressivity}; it ensures that edges can never backtrack vertically, and thereby excludes cycles. If $\varphi : G \rightarrow R$ is a vertical diagram, then we can orient the edges of $G$ using the $\geq_v$ relation. If $n$ and $m$ are nodes of $G$ connected by $e$ with $\varphi(n) \geq_v \varphi(m)$, then orient $e$ so that $e : n \rightarrow m$. Now that the edges of $G$ have been oriented, we can define the set $\inV(n)$ of incoming edges into the node $n$, and $\outV(n)$ of edges outgoing from $n$. In fact, we can order the sets of incoming and outgoing edges from a node using the notion of regular levels.

Every $u \in vR$ (or, equivalently, every number between $c$ and $d$) determines a horizontal line $[a,b] \times \{u\}$, which we will identify with $u$ so long as it doesn't cause confusion. If this line does not intersect any node of $G$, then $u$ is called a \emph{vertical regular level} of the diagram $\varphi$. For every node $n$, we can choose a vertical regular level $u$ such that $u \geq_v n$. Then every edge of $e \in \inV(n)$ will hit $u$ exactly once in a point $x_e$ by progressivity. We can then order the edges $e$ and $f$ of $\inV(n)$ by $x_e \leq_h x_f$, the order they intersect $u$. Since $\varphi$ is an embedding, this does not depend on the choice of regular level. Similarly, we can choose a regular level below $n$ to order $\outV(n)$.

Finally, given a diagram $\varphi : G \rightarrow R$, we will define $G_2$ to be $\pi_0(R - \varphi(G))$ and call them \emph{regions}. These will be interpreted as the objects of our double category. For every edge $e \in G_1$, define $\inE(e)$ to be the region containing a point $x$ such that for some $p \in e$, $x \leq_h \varphi(e)$. Similarly, define $\outE(e)$ to be the region with a point $x$ such that for some $p \in e$, $\varphi(p) \leq_h x$. These functions on edges are well defined because $\varphi$ is an embedding and by progressivity.

All these definitions should be repeated, suitably transposed, for their corresponding \emph{horizontal} notions. All that is needed is to replace $v$ by $h$ and $h$ by $v$ (and, take care, since $\geq_v$ must be replaced by $\leq_h$ and vice-versa).

\subsection{Double Graphs and Double Diagrams}

Now we come to the new definitions. A \emph{double graph} $G = vG \cup hG$ is a pair of graphs $(vG, G_0, \partial_v G)$ and $(hG, G_0, \partial_h G)$ with the same set of nodes, but with disjoint boundaries $\partial_v G \cap \partial_h G = \emptyset$, called the vertical and horizontal graphs respectively. We'll define $\partial G = \partial_v G \cup \partial_h G$.

\begin{definition}\label{Def_DoubleDiagram}
Let $G$ be a double graph and $R$ a rectangle. A \emph{double diagram} of $G$ in $R$ is an embedding $\varphi : G \rightarrow R$ such that
\begin{enumerate}
    \item The restriction $\varphi_v : vG \rightarrow R$ is a vertical diagram,
    \item The restriction $\varphi_h : hG \rightarrow R$ is a horizontal diagram,
    \item (Interpretability) For every node $n \in G_0 - \partial G$, there is a rectangle $R_n$ which contains $n$ in its interior and which satisfies the following conditions:
    \begin{enumerate}
        \item $R_n$ contains no nodes other than $n$.
        \item The only edges which intersect $R_n$ are those incident to $n$, and each edge which intersects $\partial R_n$ intersects it only once.
        \item A vertical edge which intersects $R_n$ intersects $\partial R_n$ only on its horizontal boundary. Similarly, a horizontal edge which intersects $R_n$ intersects $\partial R_n$ only on its vertical boundary.
    \end{enumerate}
\end{enumerate}
\end{definition}

The first two conditions ensure that $\varphi$ is a diagram for both the horizontal and vertical graphs of $G$. The third condition ensures that double diagrams can be interpreted in a double category. In the traditional notation, this would amount to saying that if the boundary of a rectangle is made by laying out horizontal and vertical line segments, then when considering the half of the rectangle above or below the center, no vertical line segment appears between two horizontal ones. While this is a nontrivial restriction on possible double diagrams, note that it can be ensured in practice by drawing each node as a little square so that the vertical arrows enter the top and leave the bottom of the square, and the horizontal arrows enter the right and leave the left. This gives a picture like that defining $\mu(n)$ in the following definition.

\begin{definition}
A \emph{valuation} $\mu : \varphi \to \Da$ of a double diagram $\varphi : G \rightarrow R$ in a double category $\Da$ is a bunch of functions
\begin{align*}
\mu : G_2 &\rightarrow \textbf{ob}\Da \\
\mu : vG_1 &\rightarrow \textbf{ar}_v\Da \\
\mu : hG_1 &\rightarrow \textbf{ar}_h\Da \\
\mu : G_0 - \partial G &\rightarrow 2\Da
\end{align*}
where the sets on the right above are the objects, vertical arrows, horizontal arrows, and 2-cells of $\Da$ respectively. These are required to satisfy a bunch of coherence conditions that say that all the domain and codomain conditions hold. Specifically, if $n \in G_0$ is a node and its vertical domain $\text{in}_v(n)$ consists of the edges $\{t_i\}$, vertical codomain $\text{out}_v(n)$ the edges $\{b_i\}$, horizontal domain $\text{in}_h(n)$ the edges $\{l_i\}$, and horizontal codomain $\text{out}_h(n)$ the edges $\{r_i\}$, then 
\[\begin{tikzcd}
\mu(\text{in}(l_1)) \arrow{r}{\mu l_1} \arrow{d}[swap]{\mu t_1} & \cdots \arrow{r}{\mu l_n} & \mu(\text{out}(l_n)) \arrow{d}{\mu b_1} \\
\vdots \arrow{d}[swap]{\mu t_k} & \mu(n) & \vdots \arrow{d}{\mu b_q}\\
\mu(\text{in}(r_1)) \arrow{r}[swap]{\mu r_1} & \cdots \arrow[swap]{r}{\mu r_m} & \mu(\text{out}(r_m)) 
\end{tikzcd}\]

Around the node $n$, the diagram is the dual of this square. The square above is called the \emph{signature} of $n$.

If the valuation image of $\varphi$ under the valuation $\mu$ can be composed in $\Da$, then we call this composite $\mu(\varphi)$ the \emph{value} of $\varphi$ under $\mu$. 
\end{definition}

Though a valuation certainly gives a collection of 2-cells in $\Da$, there is no guarantee that they are even compatible -- in the sense of forming a tiling of a rectangle, a pasting diagram in the usual notation -- much less composable. The compatible arrangements of 2-cells in a double category form rectangular tilings of rectangular regions of the plane. Since we have a rectangular region of the plane in mind, namely $R$, we will tile it and induce a valuation of this tiling. This will show that indeed $\mu$ gives us a compatible arrangement of double cells, and we may then inquire into its composability.

Our goal will be to show that values of double diagrams are invariant under deformation. Let's make that notion of deformation precise now, following \cite{Joyal}.
\begin{definition}
Let $G$ be a double graph. A deformation $h$ between two double diagrams $\varphi,\,\phi : G \rightarrow R$ is a continuous function $h : G \times [0,1] \rightarrow R$ such that 
\begin{itemize}
    \item $h(-,0) = \varphi$ and $h(-, 1) = \phi$,
    \item For all $t \in [0,1]$, $h(-, t) : G \rightarrow R$ is a double diagram.
\end{itemize}
\end{definition}

It is easy to see that the incoming and outgoing edges or regions do not depend on the choice of $t$ for a deformation, and that therefore valuations may be transported across deformations as well. If we have a valuation $\mu$ of $\varphi = h(-,0)$, then we'll denote by $\mu$ as well the induced valuation at $h(-,t)$. It remains to show that if $h$ deforms $\varphi$ into $\phi$, then the values $\mu(\varphi) = \mu(\phi)$ if they exist.

\subsection{Tilings and Decomposition}\label{SecTilings}
In order to show that values of double diagrams are invariant under deformation, we will show that double diagrams generate tilings that may be interpreted as a compatible arrangement of double cells (in the sense of \cite{Dawson1993}). Since, as Dawson and Pare showed in \cite{Dawson1993}, the composites of compatible arrangements of double cells are unique if they exist, and since a suitably small deformation of a diagram will fit within the same tiling, our double diagrams will have unique values that are invariant under deformation.

\begin{definition}
Given a rectangle $R$ in the plane, a \textit{tiling} $T = \{R_i \mid i \in I\}$ of $R$ is a finite set of rectangles whose union is the whole of $R$, $\displaystyle \bigcup_{i \in I} R_i = R$, and which only intersect on their boundaries, $R_i \cap R_j \subseteq \partial R_i \cap \partial R_j$ for $i \neq j$.
\bigskip

If $\varphi : G \to R$ is a double diagram and $T$ is a tiling of $R$, then $T$ is said to be \emph{admissible} for $\varphi$ if $\varphi$ restricts to a double diagram $\varphi_i$ on each $R_i \in T$ in an obvious manner, and if each $\varphi_i$ contains at most one node and one connected component of the image $\varphi(G)$. In particular, this means that no horizontal edge of $\varphi$ can intersect a vertical edge of any $R_i$, and similarly for vertical edges.
\end{definition}

\begin{proposition}\label{Lem:DiagramsTileable}
Every double diagram $\varphi : G \to R$ admits an admissible tiling $T_{\varphi}$.
\end{proposition}

\emph{Proof.} By the conditions of a double diagram, each node $n$ of $\varphi$ comes equipped with a rectangle $R_n$ only containing $\varphi(n)$ onto which $\varphi$ restricts. We will then remove $\displaystyle \bigcup_{\text{nodes } n} \interior R_n$ from the original rectangle $R$ and work there. We will cover each edge of $\varphi$, restricted to $R - \displaystyle \bigcup_{\text{nodes } n} \interior R_n$, with an $\epsilon$-tube so that no two tubes interact. Working within these tubes, we will tile each edge separately. Removing the interiors of these tiles, we are left to tile the empty portions of the diagram which can be done naively.

\comment{\begin{lemma}
If $\varphi : G \to R$ is a double diagram and $n \in G_0$ is an interior node of $\varphi$, then for every $\epsilon > 0$ there exists a rectangle $R_n$ with area less than or equal to $\epsilon^2$ containing $\varphi(n)$ such that $\varphi$ restricts to a double diagram on $\varphi\inv(R_n)$.
\end{lemma}
\begin{proof}
By condition 3 on $\varphi$, there exists a rectangle $R_n$ around $n$ onto which $\varphi$ restricts. Clearly, $\varphi$ will also restrict onto $\lambda R_n$ with $0 < \lambda < 1$ a scaling constant. 
\end{proof}}

Choose neighborhoods $R_n$ for each node $n$ of $\varphi$ so that the $R_n$ do not intersect. On the complement $R - \displaystyle \bigcup_{\text{nodes } n} \interior R_n$, each of the edges $e$ of $\varphi$ are disjoint since $\varphi$ is an embedding. Therefore, since the edges are closed and do not intersect except at the nodes, there exists an $\epsilon > 0$ that each two sets $E_e = \{r \in R \mid d(p_i(r), p_i(e)) < \epsilon\}$ for $e \in hG_1$ and $i = h$ or $e \in vG_1$ with $i = v$ are disjoint. In other words, we surround each edge by a tube which contains all points vertically or horizontally (depending on the type of edge) within $\epsilon$ of the edge, in such a way that each two tubes are disjoint. We can then tile each edge separately within this $\epsilon$-tube.

\begin{lemma}
If $e : [0,1] \to \R^2$ is a horizontally (resp. vertically) progressive curve, an ``edge'', and if $\epsilon > 0$, then $\im e$ may be covered by rectangles within the $\epsilon$-tube $E_e$, and $\im e$ will only intersect these rectangles on their vertical (resp. horizontal) edges.
\end{lemma}
\begin{proof}
We will deal with the horizontal case; the vertical one follows by symmetry. Let $e(t) = (e_1(t), e_2(t))$, and let $s = |\sup \im e_2'|$ be the largest slope of $e$ as measured against the $y$-axis. If $e$ enters a rectangle through the left vertical wall in the center, and the rectangle has width $w$, then the rectangle must have height less than $2ws$ if it is to force $e$ to leave by the right vertical wall. The vertical distance from the center of the left vertical wall to the boundary of $E_e$ is $\epsilon$ by construction, so $ws$ must be less than $\epsilon$ for the rectangle to fit in $E_e$. Thus, we may choose $w = \frac{\epsilon}{2s}$. 
$$\ingraph{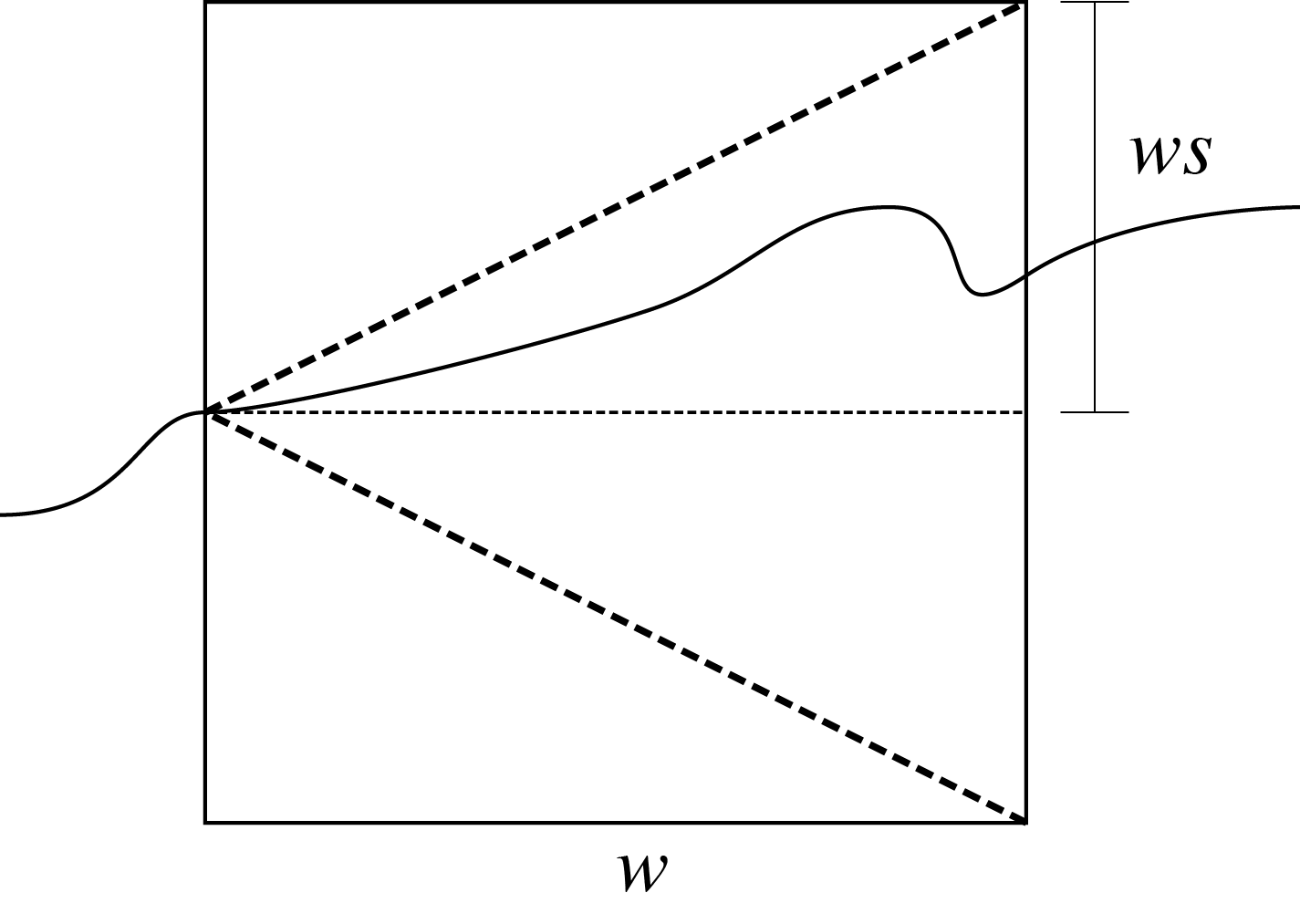}$$

Then we may freely tile $e$ with rectangles of width $w$ and height $\epsilon$ by placing one at the leftmost point on $e$ so that $e$ enters at the center of the left vertical wall, and then placing the next so that again $e$ enters at the center of the left vertical wall, and so on. When we get to the end, it is no problem to choose a smaller width, so we simply cut the rectangle down so that it fits width-wise.
\end{proof}

Now, having tiled each edge, we may remove the interior of this tiling and consider the rest of the diagram, not yet tiled. The rest of the diagram is a finite collection of connected components containing no points of $\im \varphi$, with rectilinear boundary. We may tile the rest by extending the lines of its boundary until they intersect another boundary line; this divides each region into rectangles. Having done this, we will have tiled the whole diagram. We will denote this tiling of $\varphi$ as $T_\varphi$. \qed

If $\mu : \varphi \to \Da$ is a valuation for $\varphi$ and $T$ is admissible for $\varphi$, then $T$ describes a compatible arrangement of double cells in $\Da$ as follows:
\begin{itemize}
    \item If the restriction $\varphi_i$ has a node $n_i$, then $R_i$ describes the double cell $\mu(n_i)$ with boundary given by 
    \[\begin{tikzcd}
        \mu(\text{in}(l_1)) \arrow{r}{\mu l_1} \arrow{d}[swap]{\mu t_1} & \cdots \arrow{r}{\mu l_n} & \mu(\text{out}(l_n)) \arrow{d}{\mu b_1} \\
        \vdots \arrow{d}[swap]{\mu t_k} & \mu(n_i) & \vdots \arrow{d}{\mu b_q}\\
        \mu(\text{in}(r_1)) \arrow{r}[swap]{\mu r_1} & \cdots \arrow[swap]{r}{\mu r_m} & \mu(\text{out}(r_m)) 
        \end{tikzcd}\]
        
        where $\text{in}_v(n_i) = \{t_i\}$, $\text{out}_v(n_i) = \{b_i\}$, $\text{in}_h(n_i) = \{l_i\}$, and $\text{out}_h(n_i) = \{r_i\}$ and the sequences of horizontal and vertical arrows denote their respective composites in $\Da$.
    \item If the restriction $\varphi_i$ has no node, then it must consist either of a single vertical or horizontal edge. In this case, $R_i$ describes the identity double cell on that edge, with suitable boundary.
\end{itemize}

If the compatible arrangement of double cells described by $T$ admits a composite, we will call this composite $\mu(T)$. By Theorem 1 of \cite{Dawson1993}, $\mu(T)$ is unique, independent of the way in which the arrangement is composed. We would like to define $\mu(\varphi)$ to be $\mu(T)$ for an admissible tiling of $\varphi$, but to do this we must show that $\mu(\varphi)$ is independent of the particular tiling chosen.

\begin{proposition}\label{Lem_TileInvariant}
Let $\varphi$ be a double diagram and $\mu$ a valuation for it. Suppose $S$ and $T$ are two admissible tilings for $\varphi$. Then $\mu(S) = \mu(T)$.
\end{proposition}

We will prove Proposition \ref{Lem_TileInvariant} by first showing that if $S$ is a refinement of $T$, then they have the same value, and then by showing that if $S$ and $T$ are admissible, then they have a common refinement which is also admissible.
\begin{definition}\label{Def_Refinement}
A tiling $S$ is said to be a \emph{refinement} of $T$ if for all tiles $T_i$ in $T$ there is a set of tiles $X_i \subseteq S$ of $S$ so that 
$$\bigcup_{x \in X_i} x = T_i.$$
\end{definition}

\begin{lemma}\label{Lem_InvariantRefinement}
If $S$ and $T$ are admissible tilings for $\varphi$, and $S$ is a refinement of $T$, then $\mu(S) = \mu(T)$ for any valuation $\mu$ of $\varphi$.
\end{lemma}
\begin{proof}
We will show that for each tile $T_i \in T$, the value of the tiling $X_i$ of $T_i$ is the value of $T_i$: $\mu(X_i) = \mu(T_i)$. By Theorem 1 of \cite{Dawson1993}, it will follow that $\mu(S) = \mu(T)$.

There are two cases for $T_i$; either it contains a node or it doesn't. If it doesn't contain a node, then it contains a single edge $e$ and $\mu(T_i) = \id_{\mu(e)}$. For each $x \in X_i$, $\mu(x)$ is an identity $2$-cell either of some region or of $e$. Since there is only a single edge in $T_i$, no nontrivial whiskering can take place, and the composite of identities are the respective identities. Therefore, $\mu(X_i) = \mu(T_i)$. 

If $T_i$ does contain some node $n$, then exactly one $x \in X_i$ must also contain $n$. It follows that $\mu(x) = \mu(n) = \mu(T_i)$. All the other nodes $y \in X_i$ evaluate to some identity, so $\mu(X_i)$ then equals $\mu(T_i)$. 
\end{proof}

Finally, it remains to show that there is a common admissible refinement of any two admissible tilings. Any two tilings of a rectangle have a largest common refinement which is constructed by overlaying the two tilings.
\begin{lemma} \label{Lem_RefinementAdmissable}
Let $S$ and $T$ be admissible tilings of $\varphi$. Then there is a common refinement of $S$ and $T$ which is admissible.
\end{lemma}
\begin{proof}
Let $S \# T$ denote the largest common refinement of $S$ and $T$, defined by $(S \# T)_{ij} = S_i \cap T_j$. This is not an admissible tiling because its tiles could contain more than one connected component if a tile from $S$ containing a node intersects one from $T$. However, we will show that $\varphi$ restricts to a double diagram on every tile $S_i \cap T_j$. Therefore, we can take each tile $S_i \cap T_j$ which contains more than one connected component and tile in the manner described in Lemma \ref{Lem:DiagramsTileable}. This will give an admissible common refinement of $S$ and $T$.

We need to show that the restriction $\varphi_i$ of $\varphi$ to any $R_{ij} \in S \# T$ is a double diagram. Note that $\varphi_i$ can contain at most one node and at most one connected component of $\varphi$, since $R_{ij} \subseteq S_i$ an $S_i$ is admissible. It remains to show that (1) restricted to $vG$, $\varphi_i$ is a vertical diagram, (2) restricted to $hG$, $\varphi_i$ is a horizontal diagram, and (3) that $\varphi_i$ is interpretable. 

Let $R_{ij} \in S \# T$. By construction, $R_{ij} = S_i \cap T_j$ for $S_i \in S$ and $T_j \in T$. The sides of $R_{ij}$ are parts of the sides of $S_i$ and $T_j$, and in particular the vertical boundary of $R_{ij}$ must be part of the vertical boundary of $S_i$ and $T_j$, and similarly for the horizontal boundary. We are ready to show that $\varphi_i$ satisfies the three conditions of Definition \ref{Def_DoubleDiagram}.

\begin{enumerate}
    \item Since $\varphi$ restricts to a vertical diagram on $S_i$ and is therefore vertically progressive on $S_i$, and $R_{ij} \subseteq S_i$, $\varphi$ will remain vertically progressive. Since the sides of $R_{ij}$ are parts of the sides of $S_i$ and $T_j$, the boundary nodes of the domain of $\varphi_i$ must lie in one of the boundaries of $S_i$ and $T_j$. Since $S$ and $T$ are admissible, this means that the boundary nodes must lie only on the vertical boundaries of $S_i$ or $T_j$, and so must lie on the vertical boundary of $R_{ij}$. Furthermore, if $\varphi_i(x) \in \partial R_{ij}$, then it is also in $\partial S_i$ or $\partial T_j$, and is therefore a boundary node of $S_i$ or $T_j$ and so a boundary node of $R_{ij}$. 
    \item This condition may be shown to hold by the same argument as above, replacing ``vertical'' for ``horizontal''.
    \item This condition concerns interior nodes. If $R_{ij}$ does not contain an interior node, then it is trivially satisfied. If $R_{ij}$ contains $n$, then we will let $R_n = R_{ij}$. Then, $R_n = R_{ij}$ must contain only one node; otherwise, since $R_{ij} \subseteq S_i$, $S_i$ would contain more than one, contradicting the admissibility of $S$. Since each edge $e$ which intersects $R_{ij}$ intersects $S_i$, it can intersect $R_{ij}$ only once. Since $S_i$ contains only one connected component, $e$ must be incident to $n$. Finally, if $e$ is vertical then it must intersect $S_i$ and $T_j$ on a vertical side, and therefore $R_{ij}$ on a vertical side (which might be part of either $S_i$ or $T_j$'s vertical sides). Similarly, if $e$ is horizontal, it must intersect $S_i$ and $T_j$ on a horizontal side, and so $R_{ij}$ on a horizontal side. 
\end{enumerate}
\end{proof}

Using these two lemmas, we can quickly prove Proposition \ref{Lem_TileInvariant}

\begin{proof}[Proof of Proposition \ref{Lem_TileInvariant}]
By Lemma \ref{Lem_RefinementAdmissable}, there is a common refinement $X$ of $S$ and $T$ is admissible. By applying Lemma \ref{Lem_InvariantRefinement} twice, we see that $\mu(S) = \mu(X) = \mu(T)$, which was to be shown.
\end{proof}

Therefore, we can safely define the value $\mu(\varphi)$ of a double diagram to be $\mu(T)$ for some admissible tiling $T$ of $\varphi$, knowing that $\mu(\varphi)$ is invariant under our choice of such a tiling. We are now ready to prove that the value of a double diagram is invariant under deformation, given the existence of admissible tilings.

\begin{lemma} \label{Lem_ValueInvariant}
Let $h : G \times [0,1] \to R$ be a deformation of double diagrams and $\mu$ a valuation of $h(-, 0)$. Suppose that the value $\mu(h(-,0))$ exists. If for all $t \in [0, 1]$, there is an admissible and composable tiling $T_t$ for $h(-,t)$, then $\mu(h(-,t))$ exists and equals $\mu(h(-,0))$.
\end{lemma}
\begin{proof}
Note any suitably small perturbation of a diagram preserves the admissibility of any tiling of it. Therefore, by the compactness of the interval, we may choose finitely many $0 = t_0$, $\ldots$, $t_n = 1$ so that the assumed tiling $T_i$ of $h(-, t_i)$ is also admissible for $t_i+1$, and so that all $t \in [0, 1]$ are admissibly tiled by some $T_i$. We can then transport the valuation $\mu$ from $h(-, t_0)$ to $h(-, t_1)$, and since $T_0$ is admissible and composable for both, $\mu(h(-, t_0)) = \mu(T_0) = \mu(h(-, t_1))$. Continuing this way, we see that the value is invariant at all stages of the deformation.
\end{proof}

It remains to show that there is an admissible and composable tiling for any diagram that describes a composable arrangement of double cells.

\subsection{Neat Diagrams and Composable Tilings}

Not all tilings of the plane by rectangles are composable by horizontal and vertical mergings of rectangles; therefore, not all arrangements of 2-cells in a double category are composable with the two sorts of compositions. In the above section, we naively tiled a double diagram. In this section, we'll see that under certain conditions on the diagram, this tiling may be modified in order to be composable.

In his paper \cite{Dawson}, Dawson characterizes those diagrams in the usual notation for double categories which admit composites by repeated application of the two binary compositions available in a double category. To do this, he uses the tile-order machinery developed by him and Par{\'e} in \cite{Dawson1993}. A tile-order is an order theoretic abstraction of a tiling of a rectangle; this abstracts the usual notation since such tilings are given by any array of 2-cells in the usual notation. 

Let's recall the language of tile-orders. 
\begin{definition}\label{Def-Tileorder}
A tile order $A$ is a set admitting two orders, ``below'', and ``beside'', which can be realized as the set of tiles in a tiling of a rectangle given the two orders as transitive closures of the relations ``$T$ is below $T'$ when the top edge of $T$ intersects the bottom edge of $T'$'' and ``$T$ is beside $T'$ when the right edge of $T$ intersects the left edge of $T'$'', respectively.
\end{definition}

Given a tile order $A$, we can draw its Hasse diagram in any tiling which realizes it by connecting the centers of tiles to each other with the right sort of line (if the line passes through a vertical edge, it is labeled horizontal, and vice versa). Note that, therefore, any composite in a double category may be expressed by a double diagram. First, express the composite in the usual notation; this gives a tiling of a rectangle. Draw a node in each tile, and draw a line through each boundary line of the tile of the correct orientation (i.e. a vertical edge through a line on the top boundary of a tile, etc). This gives a double diagram together with an admissible tiling of it; therefore, the composite of the double diagram exists and equals the composite we wished to express. 

In \cite{Dawson}, Dawson introduces the notion of a neat tile order, and shows that tilings are composable if and only if they induce neat tile orders. 
\begin{definition}
A tile order is \emph{neat} if for any $\epsilon > 0$ it can be realized as the tile order of a tiling whose Hasse diagram has each line marked vertical being within $\epsilon$ of being actually vertical, and each line marked horizontal being within $\epsilon$ of being actually horizontal.
\end{definition}

\begin{theorem}{(Dawson, \cite{Dawson})}
A tiling is composable if and only if its induced tile order is neat.
\end{theorem}

This is because the only obstruction to the composition of a diagram is a pinwheel, drawn as follows as both a tiling and a double diagram. No two tiles in a pinwheel may be merged, so it cannot be composed.
$$(\ast) \quad \ingraph{full_pinwheel}$$

Any tiling which runs into a pinwheel at some process of its composing cannot be fully composed, at least by that method of composition. Dawson shows further that the pinwheel (and its reflection) are the only such obstructions to composition.

\begin{proposition}{(Prop 4.2 in Dawson, \cite{Dawson})}
If a tiling may not be composed by some process of composition, then that process induces a pinwheel.
\end{proposition}

 Since the pinwheel is the only obstruction to the composability of a tile-order, we will ask that $\varphi$ lacks pinwheels. More precisely, we will ask that $\varphi$ not have a \emph{full pinwheel} as an induced subgraph. The full pinwheel is drawn as $(\ast)$, and is defined as a pinwheel of tiles for which there is a horizontal edge through each vertical interior wall, and a vertical edge through each horizontal interior wall.

There is also a version of the pinwheel with opposite ``orientation''. All results about pinwheels follow for both by symmetry. 

\begin{lemma} \label{Lem_NoPinwheelsComposable}
If $\varphi$ does not have any induced full pinwheels, then it admits a composable tiling.
\end{lemma}
\begin{proof}
We will show that any pinwheels which arise in the construction of the naive tiling $T_\varphi$ are inessential in the sense that they can be replaced without affecting the value of $T_\varphi$. We note that if any of the rectangles (excluding the central square) is divided in half along the line which extends its intersection with the interior square, then the pinwheel may be composed. For that reason, it suffices to show that one of the outer rectangles of the pinwheel may be divided.

If $\varphi$ does not have any induced full pinwheels, then if $T_\varphi$ contains a pinwheel, it must be missing one of the connecting edges (otherwise, it would be full). We may assume, furthermore, that if the lines comprising the rectangles are extended, they do not intersect any node. If this occurs, we may perturb the diagram slightly so as to avoid it. We then consider two cases: first where the missing edge is an outer edge going between two outer rectangles, and second when it is an inner edge connecting an outer rectangle with the central square. Each particular instance of these two cases is the same as any other up to symmetry.

We suppose that we are in some stage of composition of the valuation of the tiling $T_\varphi$. In the first case, suppose the missing edge would connect the upper two outer rectangles. The node in the upper right rectangle is either above or below the bisecting line. In the case that it is below, we may simply extend the bisecting line.
$$\ingraph{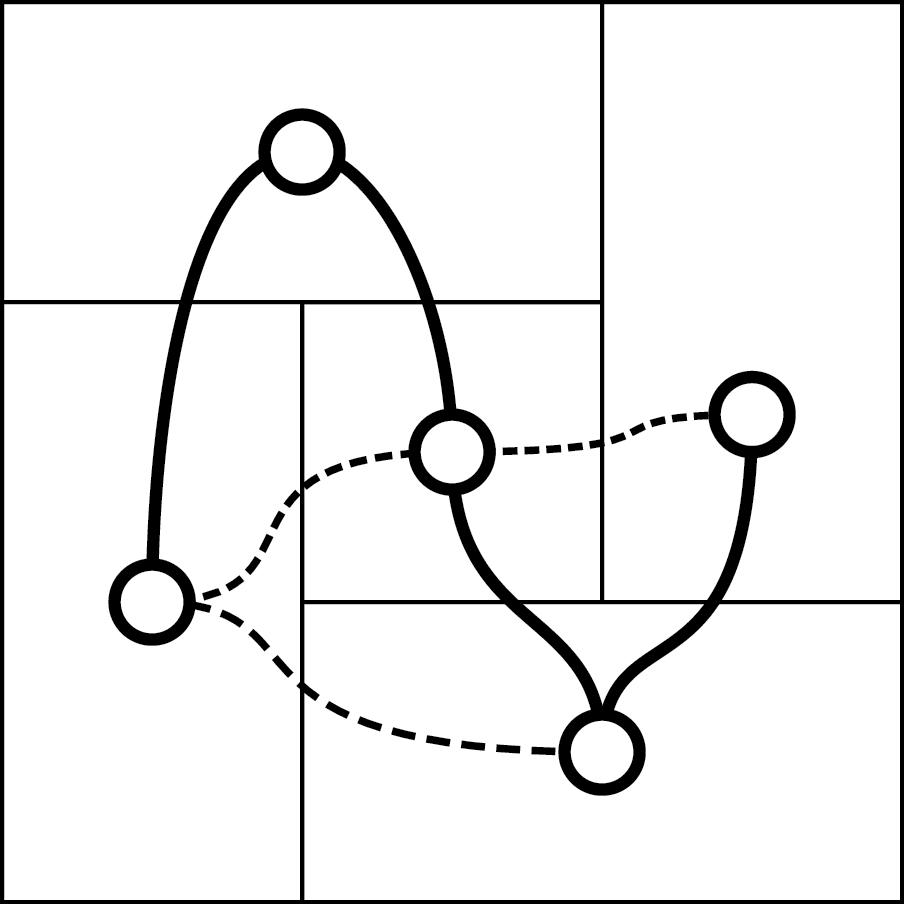} \quad \longrightarrow \quad \ingraph{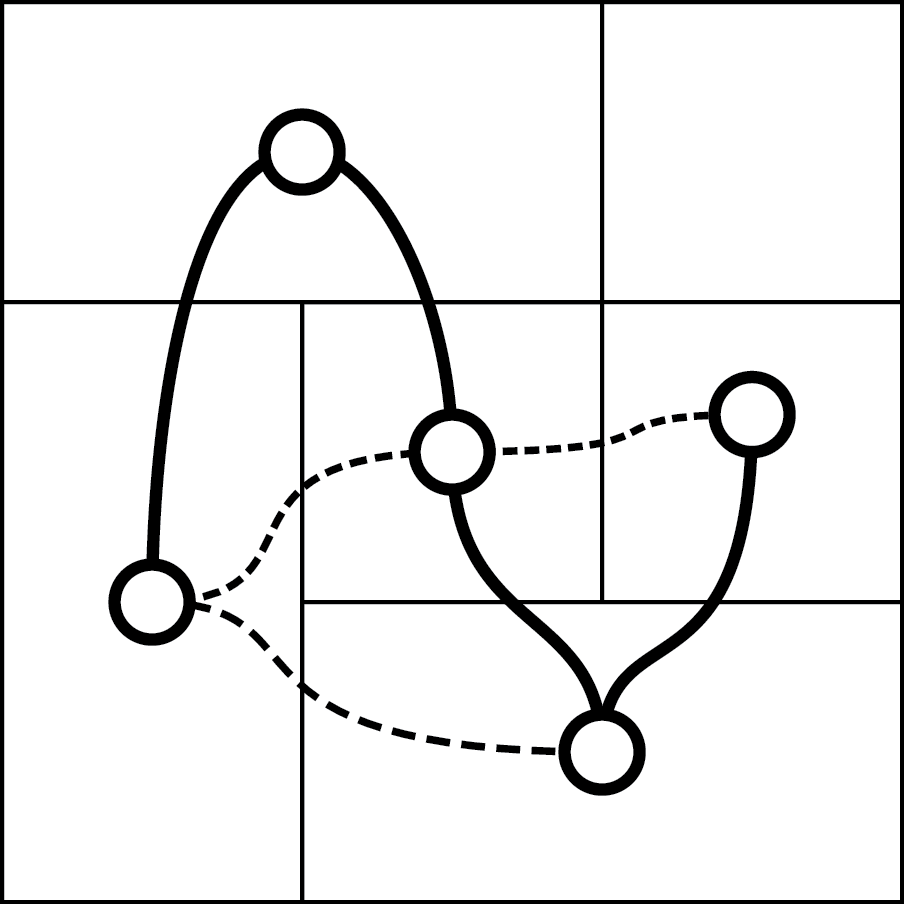}$$

If the node is above the bisecting line, then we may deform $\varphi$ so that it is below without leaving the tiling. There is a possible obstruction to this action, however, in the case that there is a further horizontal edge which leaves the top right node to the right, and which exits the top right rectangle above the bisecting line.
$$\ingraph{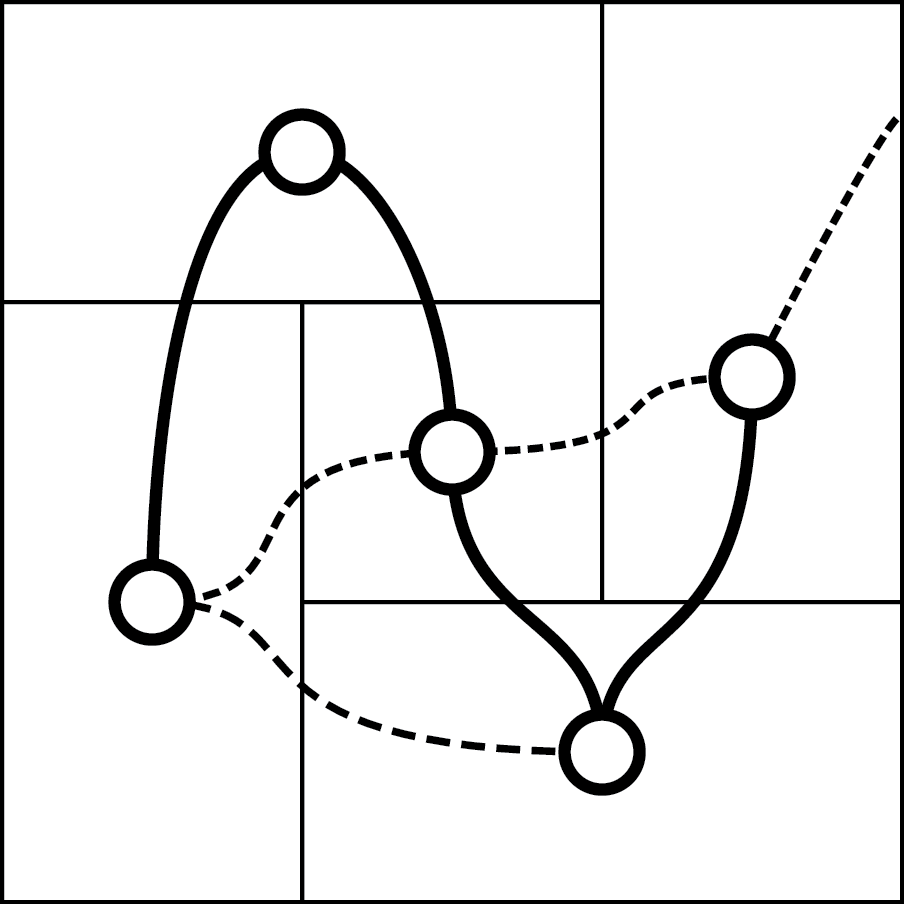}$$

In this case, we cut a suitably small rectangle from the right side of the pinwheel (small enough that it doesn't intersect any vertical edges which happen to be around), and then deform the diagram so that the horizontal edge leaves below the bisecting line.
$$\ingraph{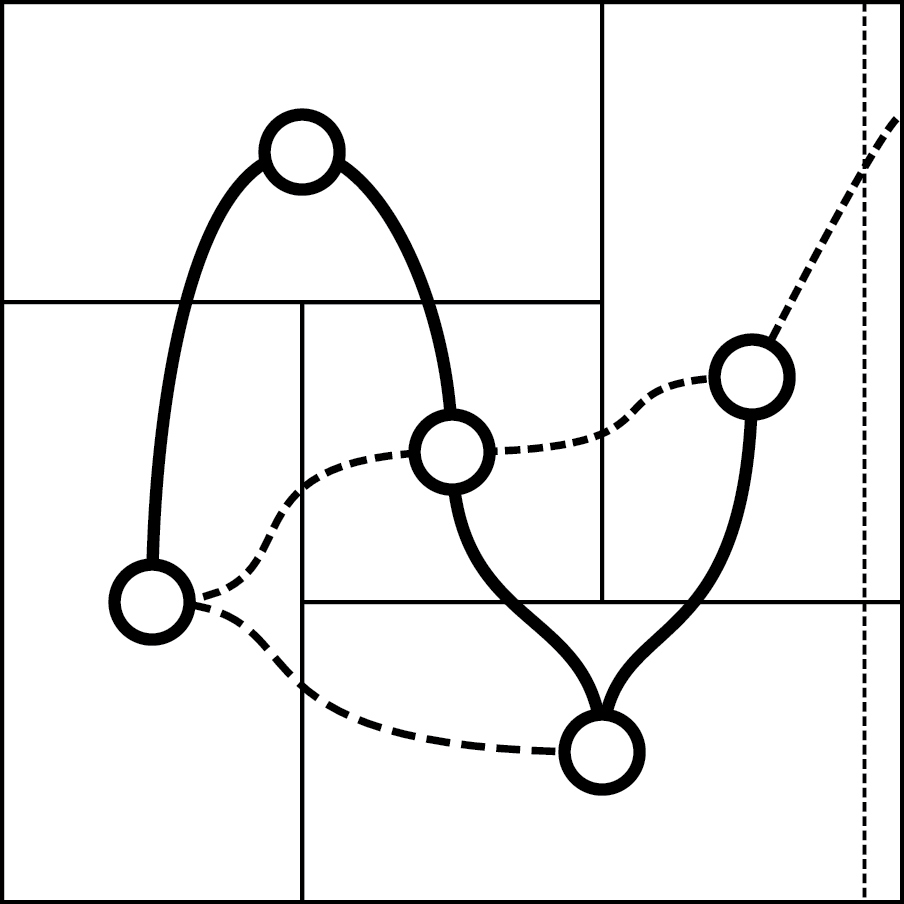} \quad \longrightarrow \quad \ingraph{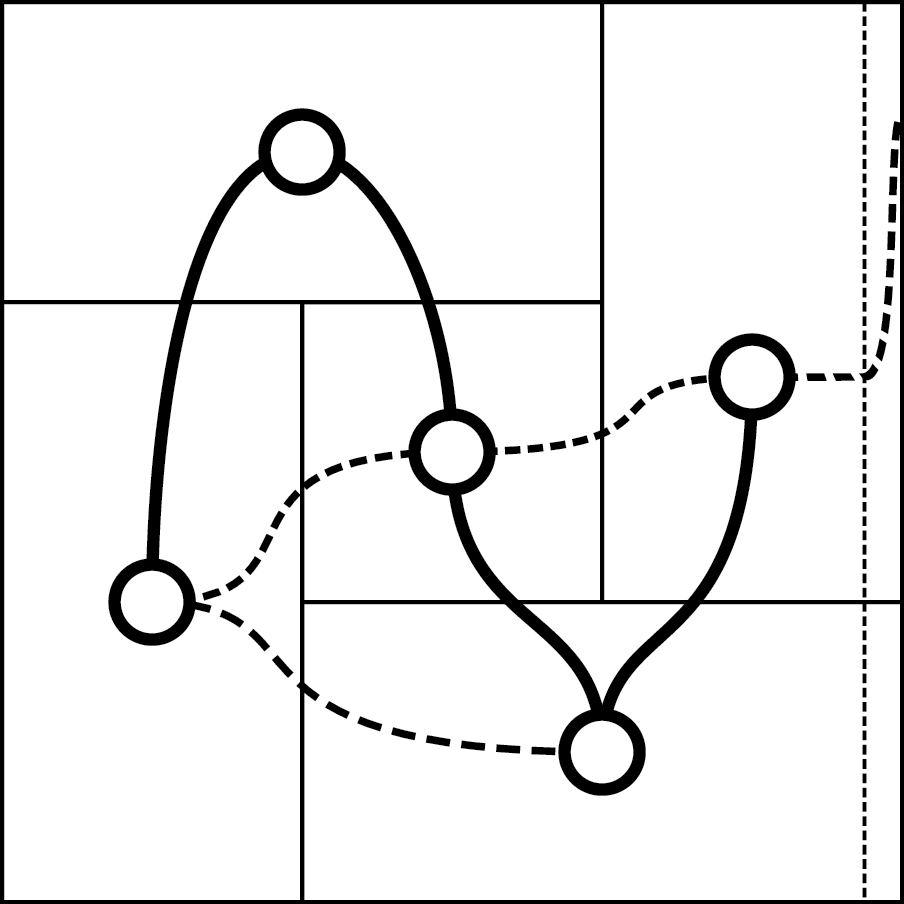} \quad \longrightarrow \quad \ingraph{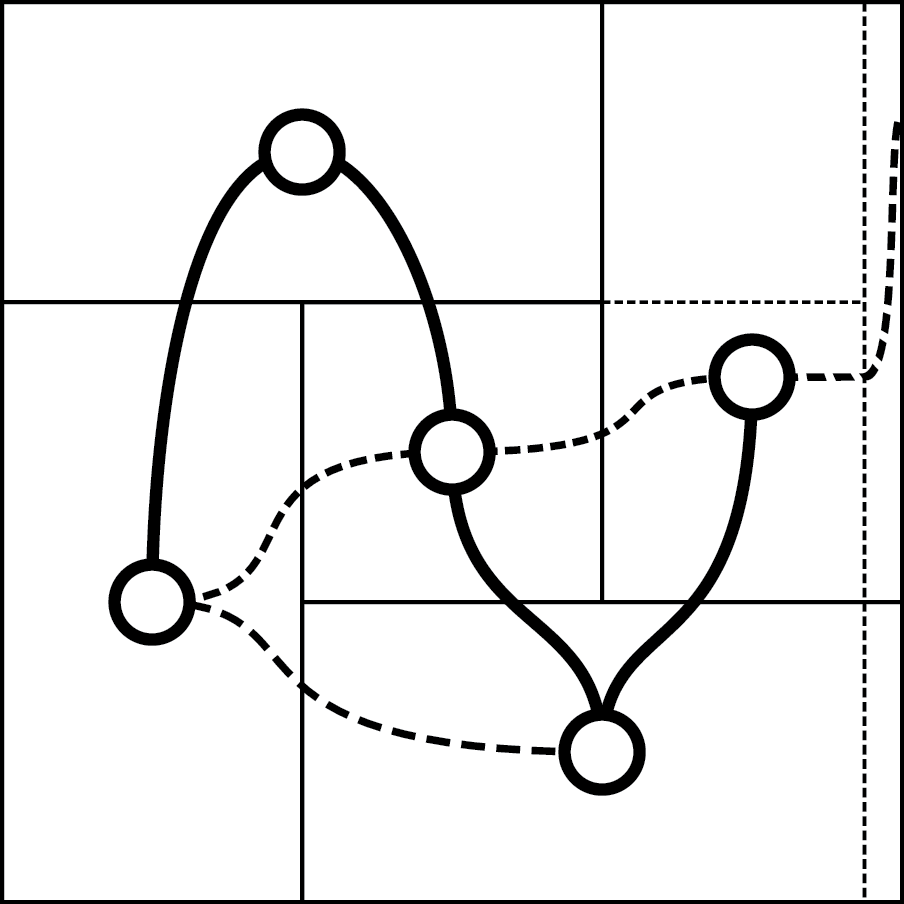} $$

We may then compose the part of the pinwheel which is left, compose the top and bottom parts of the cut, and then compose the resulting composites. Therefore, even with such an obstruction, we may still compose.

In the second case, suppose that the missing edge would connect the upper left rectangle to the center square. Either the node in the upper left rectangle is to the left or to the right of the bisecting line. If it is to the left, we may simply extend the bisecting line, as shown here:
$$\ingraph{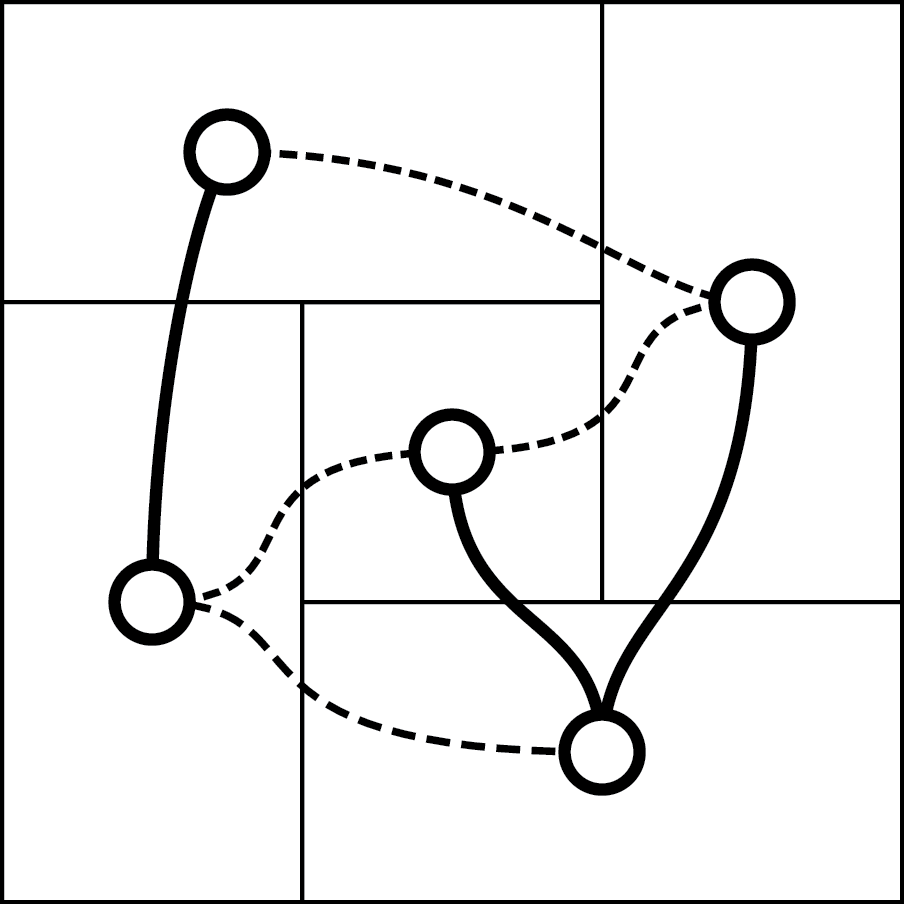} \quad \longrightarrow \quad \ingraph{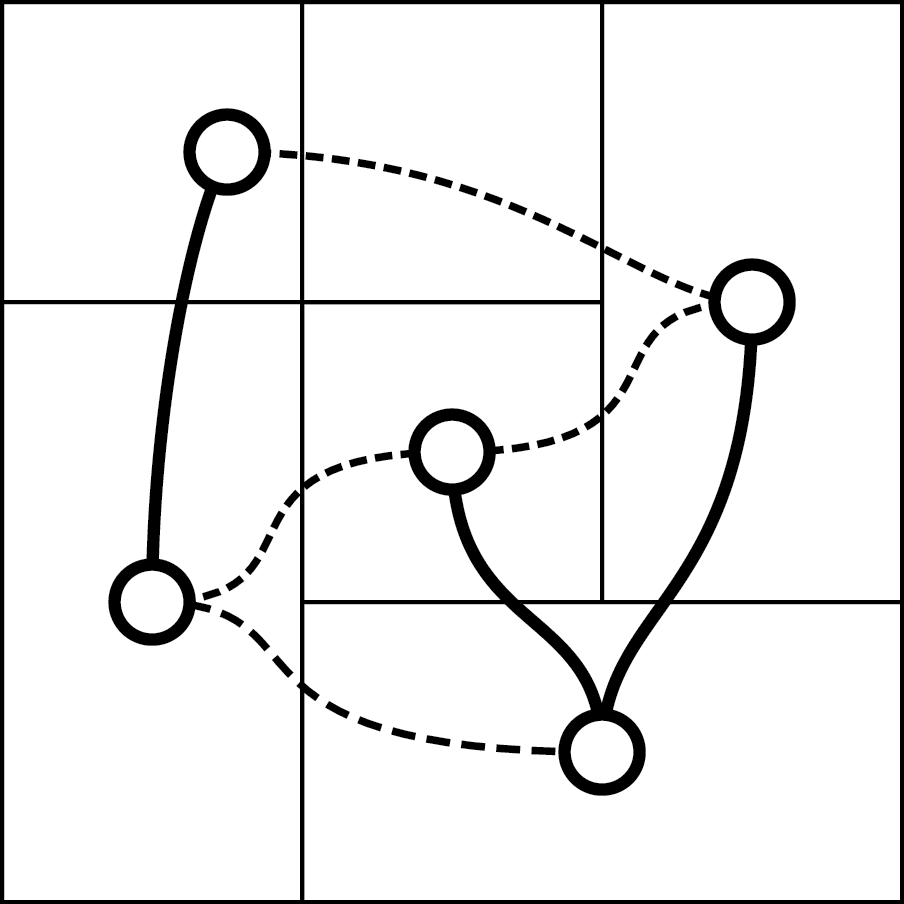}$$
If the node is to the right of the bisecting line, then we may deform $\varphi$ so that it lies to the left without leaving the tiling. If an obstruction occurs, we proceed as we did in the other case.

It is clear that in all the above cases, the value of the diagram has not changed. Since the only obstruction to the composability of a tiling is the pinwheel (\cite{Dawson}), and we have shown that we can cut each pinwheel so that it can be composed, the tiling is composable.
\end{proof}

There is a more natural characterization of full-pinwheel-free double diagrams which we define by analogy to the notion with the same name in \cite{Dawson}.
\begin{definition}
A double diagram $\varphi = (\varphi_1, \varphi_2) : G \to R$ is \emph{neat} if for every $\epsilon > 0$, there exists a double diagram $\varphi_\epsilon : G \to R$ and a deformation $h_\epsilon$ of $\varphi$ into $\varphi_\epsilon$, such that
\begin{enumerate}
    \item On each vertical edge $e \in vG_1$, $|\varphi_1'(x)| < \epsilon$ for all $x \in e$, and
    \item on each horizontal edge $e \in hG_1$, $|\varphi_2'(x)| < \epsilon$ for all $x \in e$.
\end{enumerate}
\end{definition}

Intuitively, a diagram is neat if it may be deformed so that the vertical edges are within $\epsilon$ of actually being vertical, and similarly for the horizontal edges, for any desired $\epsilon >0$. For the most part, neatness may be noted immediately by looking at a diagram. In general, neatness will simply follow from the ``good practice'' of writing vertical arrows as vertical as possible, and horizontal arrows as horizontal as possible. Every neat diagram is composable.

\begin{lemma}\label{Lem_NeatIffComposable}
If double diagram is neat, then it does not induce any full pinwheels and is therefore composable.
\end{lemma}
\begin{proof}
By inspection, we see that the pinwheel is not neat, so a neat diagram cannot induce a pinwheel. 
\end{proof}

I expect that a diagram which does not induce full pinwheels is neat, as is the case for tile orders. We are ready now to prove the full invariance under deformation for double diagrams.

\begin{theorem}[Invariance of Value under Deformation] \label{Thm:InvariantValue}
Let $h : G \times [0,1] \to R$ be a deformation of double diagrams and $\mu$ a valuation of $h(-, 0)$. Suppose that $h(-,0)$ is neat so that the value $\mu(h(-,0))$ exists. Then $\mu(h(-,t))$ exists and equals $\mu(h(-,0))$ for all $t$.
\end{theorem}
\begin{proof}
Since $h(-,0)$ is neat, $h(-,t)$ is neat for all $t$. Therefore, by Lemma \ref{Lem_NeatIffComposable}, each $h(-,t)$ does not induce any full pinwheels, and so by Lemma \ref{Lem_NoPinwheelsComposable}, it admits a composable tiling. Then the hypotheses of Lemma \ref{Lem_ValueInvariant} are satisfied, so the desired conclusion follows.
\end{proof}

\section{Equipment Diagrams}\label{Sec_EquipmentDiagrams}

Now that we have seen that the value of a composable double diagram is invariant under deformation, we turn our attention to equipment diagrams. Equipment diagrams differ from double diagrams in that they allow the bending of vertical edges left and right. Bending left and right correspond to moving to the companion or conjoint respectively of a vertical arrow via a (co)unit. In this section, we will show that value of an equipment diagram is invariant under deformation \emph{up to the insertion of the correct (co)units}. We will proceed with this argument in several steps.
\begin{itemize}
    \item In Section 3.1, we look more closely at the algebra of companions and conjoints in preparation for the definition and proof of invariance.
    \item In Section 3.2, we define an equipment diagram and the notion of an induced valuation. We then prove functoriality for induced diagrams (Lem \ref{Lem_InducedValueFacts}).
    \item In Section 3.3, we show that every equipment diagram admits an admissible tiling and show that the value of an equipment diagram is invariant under deformation (Theorem \ref{Thm:InvariantValueEq})
\end{itemize}

\subsection{Companions, Conjoints, and Bends}\label{Sec_Bends}

In Section \ref{Sec_SpiderLemma}, we saw how nodes could be slid around bends in vertical wires through the definition of the conjoint (resp. companion) \emph{mate}. That is, if we define the conjoint mate $\ingraph{h2CellRYr}$ of $\ingraph{v2CellRY}$ to be $\ingraph{ZigZag2Cell_Rh_RY}$, then we can then slide beads around bends thanks to the following equations:
$$\ingraph{Mate1} = \ingraph{Mate3} = \ingraph{Mate2}.$$

Although the same node appears on the same edges in both the left and right of this equation, they mean different things. If we had a value of $\tinygraph{v2CellRY}$ in mind then we could use it to evaluate $\tinygraph{h2CellRYr}$ by the definition above, and if we had a value of $\tinygraph{h2CellRYr}$ in mind we could use it to define $\tinygraph{v2CellRY}$ by bending in the other way. But the value of $\tinygraph{v2CellRY}$ and $\tinygraph{h2CellRYr}$ are not equal in general; if one of the wires is not evaluated to an identity, then they won't even have the same signature and so can't be compared for equality. This situation is reflected generally in the Spider Lemma (Lemma \ref{Lem_SpiderLemma}).

Since a deformation of an equipment diagram can change the signature of the nodes, we can't simply transport a valuation accross a deformation as we did for double diagrams. We have to keep track of the companion and conjoint (co)units to add in so that the Spider Lemma holds. To make sure we know what the right (co)units to add are, we embark on a brief study of their algebra.

Consider the situation of a single vertical wire. This may have any of three orientations:
\begin{center}
Down (arrow): $\midgraph{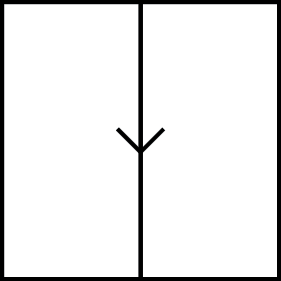}$, \quad
Right (conjoint): $\midgraph{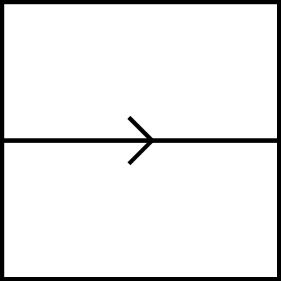}$, and \quad
Left (companion): $\midgraph{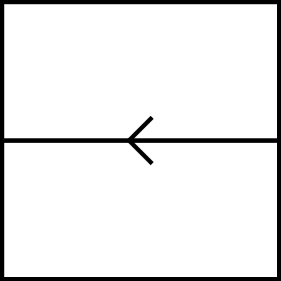}$.
\end{center}

There are four bends which change direction:
\begin{center}
$\midgraph{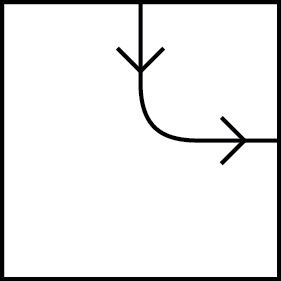}, \quad \midgraph{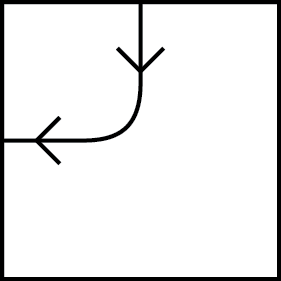}, \quad \midgraph{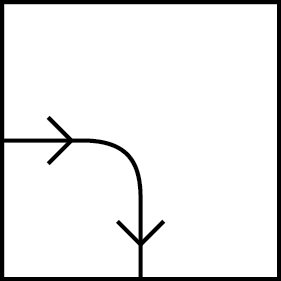}, \quad \midgraph{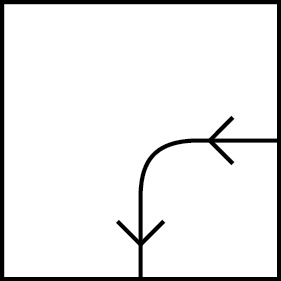},$
\end{center}
which correspond to the respective units and counits. Note that each two of these bends may be composed along the edge in exactly one way (preserving the orientation of the edge). Therefore, we get the small category $\Ka$ of possible composites shown below.
\[
\begin{tikzcd}[column sep = huge]
\midgraph{Arrow_L} \arrow[bend left = 25]{r}{\midgraph{BendArrow_RD}} & \midgraph{Arrow_D} \arrow[bend left = 25]{r}{\midgraph{BendArrow_DR}} \arrow[bend left = 25]{l}{\midgraph{BendArrow_DL}} & \midgraph{Arrow_R} \arrow[bend left = 25]{l}{\midgraph{BendArrow_LD}}
\end{tikzcd}
\]

The kink identities make this category into a groupoid with $\midgraph{BendArrow_RD}\inv = \midgraph{BendArrow_DL}$ and $\midgraph{BendArrow_LD}\inv = \midgraph{BendArrow_DR}$. We will write composition in diagrammatic order. Undrawn above are the composites $\midgraph{BendArrow_RD} \cdot \midgraph{BendArrow_DR} = \midgraph{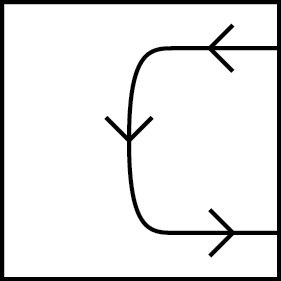}$ and $\midgraph{BendArrow_LD} \cdot \midgraph{BendArrow_DL} = \midgraph{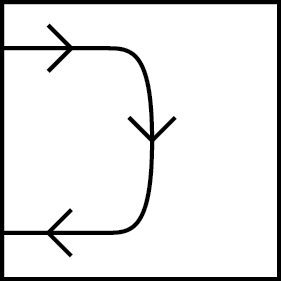}$. It is quick to verify that these and those drawn above are the only morphisms in $\Ka$. In terms of equipments, this means that a composite of (co)units of a vertical arrow must equal one of the bends described by a morphism in $\Ka$.

We will replace $\Ka$ by an isomorphic groupoid that will be easier to refer to and reason with. The groupoid $\Ka'$ will be generated by the data drawn below,
\[
\begin{tikzcd}[column sep = large]
\ell \arrow[loop left]{l}{r\ell} \arrow[bend left = 25]{r}{rd} & d \arrow[out=-105,in=-75, loop,distance=0.5cm,swap]{}{dd} \arrow[bend left = 25]{l}{d\ell}
\arrow[bend left = 25]{r}{dr} & r \arrow[loop right]{r}{\ell r} \arrow[bend left = 25]{l}{\ell d}
\end{tikzcd}
\]
subject to the relations $(rd)\inv = (d\ell)$ and $(dr)\inv = (\ell d)$, with $r\ell$, $dd$ and $\ell r$ denoting the respective identities. We can read the morphism $xy$ as the line going from the $x$ boundary to the $y$ boundary. The object labels are the three direction $\ell$, $d$, and $r$ seen as the objects of $\Ka$. The morphism labels come from watching the way the edge flows through the sides of the square, with both the top and bottom of the square labeled as $d$. If we further write $rr : \ell \to r$ and $\ell \ell : r \to \ell$ as the unique composites with those signatures respectively, then we can describe composition in $\Ka'$ in a rather simple way. 

For a direction $x \in \{\ell, d, r\}$, let $\bar{x}$ be the opposite direction. That is, $\bar{d} = d$, $\bar{\ell} = r$, and $\bar{r} = \ell$. There is a unique morphism $x \to y$ in $\Ka'$ and it is labeled $\bar{x}y$. Since $\bar{x}y \cdot \bar{y}z = \bar{x}z$, we see that composition in $\Ka$ is given by the relations $y\bar{y} = (\,\,)$ on (suitable) words in $\{\ell, d, r\}$. We can package this reasoning into the following lemma.

\begin{lemma}\label{Lem_WordsCompose}
Let $(w_{i,1} w_{i,2})_{1 \leq i \leq k}$ be a sequence of two letter words in the alphabet $\{\ell, d, r\}$ so that $\overline{w_{i,2}} = w_{(i+1),1}$ for each $1 \leq i < k$. Then $w = \prod_i  w_{i,1} w_{i,2}$ denotes a morphism in $\Ka$ which equals $w_{1,1}w_{k,2}$.
\end{lemma}

Since $\Ka$ and $\Ka'$ are isomorphic, we will refer to the bends of $\Ka$ by the labels in $\Ka'$, and equate the two groupoids without confusion. We now have the tools necessary to define and reason about equipment diagrams and their valuations.

\subsection{Equipment Diagrams and Valuations}\label{Subsec_EquipmentDiagrams}

In the previous section, we saw how the value of an equipment diagram could change (in a structured way) through deformation. But in order to know how to compose the new value of a node with the correct (co)units, we need to know how its signature has changed from before the deformation. Therefore, the data of an equipment diagram must contain the signatures of its nodes. For this reason, we will consider the rectangles $R_n$ which surround each node to be part of the \emph{structure} of an equipment diagram; embeddings with a different choice of node-rectangles will yield different diagrams. 

\begin{definition}\label{Def_EquipmentDiagram}
An \emph{equipment diagram} is a pair $(\varphi, \Ra)$ of an embedding $\varphi : G \to R$ together with a choice $\Ra$ of rectangles $R_n \subseteq R$ for each node $n \in G_0 - \partial G$ such that 
\begin{enumerate}
    \item $\varphi_h : hG \to R$ is a horizontal diagram,
    \item $\varphi_v : vG \to R$ satisfies the following conditions:
    \begin{enumerate}
        \item $\varphi(\partial \,vG) \subseteq \interior \partial R$, and
        \item \label{EquipmentC1} For all edges $e \in vG_1$, if $t_0 < t_1$ and $e(t_0) =_h e(t_1)$, then $e(t_0) <_v e(t_1)$, and
        \item For all edges $e \in vG_1$, if $e(0)$ or $e(1)$ is in $\partial R$, but they are not on opposing vertical boundaries of $R$ in $\partial_v R$, then $e(0) <_v e(1)$.
    \end{enumerate}
    \item The \emph{node-rectangles} $R_n \in \Ra$ contain $n$ in their interior, and satisfy the following.
    \begin{enumerate}
        \item $R_n$ contains no nodes other than $n$.
        \item $R_n$ contains a single connected component of the image of $\varphi$.
        \item Restricted to $R_n$, $\varphi$ satisfies conditions $(1)$ and $(2)$ above.
    \end{enumerate}
\end{enumerate}
\end{definition}

The progressivity of the vertical diagram $\varphi_v$ has been replaced by two conditions ensuring that the vertical edges never flow upwards. The condition \ref{EquipmentC1} ensures, in particular, that an edge may never fully loop around a node, and 2c ensures that a vertical edge may never `flow upwards'. As with double diagrams, the third condition ensures that equipment diagrams may be interpreted in an equipment by making the signature well defined. Again, it may be ensured by drawing each node as a small square, where the horizontal edges must pass through the left and right sides, and the vertical edges can pass through any side as long as they are pointing the correct way.

It will be useful to know the direction of flow of a vertical edge $e$ through a rectangle $S \subseteq R$. First, we will label the sides of any rectangle with flow directions $\{\ell, d, r\}$ with the top and bottom sides being labeled $d$ and the left and right sides being labeled $\ell$ and $r$ respectively. We may then make the following definition.

\begin{definition}
Let $(\varphi : G \to R, \Ra)$ be an equipment diagram, let $S \subseteq R$ be a rectangle, and let $e$ be a vertical edge which intersects $S$. Then the \emph{direction of flow of $e$ within $S$}, $d_{\varphi|S}(e) \in \Ka$ (or $d_S(e)$ when $\varphi$ is apparent), is $d^0_{\varphi|S}(e)d^1_{\varphi|S}(e)$ for $d^i_{\varphi|S}(e) \in \{\ell, d, r\}$ where
\begin{itemize}
    \item if $S$ contains no node, then $d^0_{\varphi|S}(e)$ is the side through which $e$ enters $S$, and $d^1_{\varphi|S}(e)$ is the side through which $e$ leaves $S$,
    \item if $S$ contains a node $n$ and $e(0) = n$, then $d^1_{\varphi|S}(e)$ is the side through which $e$ leaves $S$ and $d^0_{\varphi|S}(e) = \overline{d^1_{\varphi|S}(e)}$, and
    \item if $S$ contains a node $n$ and $e(1) = n$, then $d^0_{\varphi|S}(e)$ is the side through which $e$ enters $S$ and $d^1_{\varphi|S}(e) = \overline{d^0_{\varphi|S}(e)}$.
\end{itemize}

In other words, the direction of flow of $e$ in $S$ is the morphism of $\Ka$ whose domain is opposite to the side $e$ enters $S$ and whose codomain is the side $e$ leaves $S$, with edges incident to nodes assumed to be straight. This defines a graph homomorphism from $G$, divided at its intersection with the tiling, to $\Ka$.
\end{definition}

The direction of flow $d_{\varphi|S}(e)$ is the direction that $e$ is flowing through $S$ if it contained at most a single bend within $S$. Note that if $S$ contains a node, then $e$ is considered to be unbent in $S$.

We can now define the signature of a node $n$ in $(\varphi, \Ra)$ much as we did in the case of double diagrams, but keeping track of the possibility that vertical edges may be flowing horizontally into $n$. That is, we define the vertical and horizontal inputs and outputs of $n$ to be the set of edges intersecting the respective sides of $R_n$, ordered left to right and top to bottom. This time, however, vertical edges $e$ may enter or leave through the vertical boundary of $R_n$. Since $R_n$ is a rectangle containing a node, we see that $d_{\varphi|R_n}(e)$ is an identity arrow of $\Ka$ for any vertical $e$. In the sets of inputs and outputs, we put in the companion ($r \ell$), arrow ($dd$), or conjoint ($\ell r$) of $e$ according to $d_{\varphi|R_n}(e)$.

A \emph{valuation} $\mu$ for an equipment diagram $(\varphi, \Ra)$ is a collection of functions assigning the parts of $\varphi$ to objects, arrows, and cells in an equipment $\Ea$ so that the signature of the value of a node consists of the values of its signature (as was the case with valuations of double diagrams). We would like to define the composite of the images of nodes of $\mu$ to be the \emph{value} $\mu(\varphi, \Ra)$, but as before there is no guarantee that the image denotes a compatible arrangement in $\Ea$. In fact, this time the situation is worse, since vertical edges may bend as they move between nodes, meaning that we will need to add in the correct (co)unit of the conjunction/companionship. To do this, we will again turn to tilings. But first we must discuss the way valuations change under a change in choice of node-rectangles.

As discussed in Subsection \ref{Sec_Bends}, if we know the value of a node with a vertical edge in a particular position in its signature, then we can deduce the value of the node with that vertical edge in a different position in its signature via the Spider Lemma. Since the choice $\Ra$ of node-rectangles in an equipment $(\varphi, \Ra)$ encodes the desired ``original signature'' of the nodes of $\varphi$, a change in choice of node-rectangles should induce a structured change in valuation. Consider, as a paradigmatic example, the simple equipment diagram
$$\fullgraph{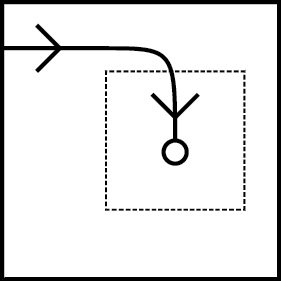},$$
where the node is $n$, the vertical edge is $e$, the outer rectangle is $R_n$, and the inner rectangle is $S_n$. Suppose we have a valuation $\mu$ of $(\varphi, \Ra)$, and we are attempting to deduce a valuation $\mu^{\Ra}_{\Sa}$ of $(\varphi, \Sa)$. Then we know the value $\mu(n)$ of $n$ where $e$ is considered to enter it from left as a conjoint; that is, $d_{R_n}(e) = \ell r$. But $e$ enters $n$ from the top in $S_n$, $d_{S_n}(e) = dd$, so it will appear as an arrow in this signature. The bend in the picture, which turns the rightward flowing $e$ in $R_n$ into a downward flowing $e$ in $S_n$, is $\ell d = d^0_{R_n}(e)\overline{d^0_{S_n}(e)}$, flowing from the direction $e$ enters $R_n$ to the direction opposite that from which $e$ enters $S_n$. In order to satisfy the Spider Lemma, $\mu^{\Ra}_{\Sa}(n)$ must therefore involve the opposite bend, so that the composite of the two will equal $\mu(n)$. We formalize this now.

\begin{definition}
Let $n$ be a node in $G$, $e$ a vertical edge incident to it, and let $(\varphi : G \to R, \Ra)$ and $(\varphi, \Sa)$ be equipment diagrams. The \emph{bend in $e$ from $R_n$ to $S_n$}, $\Delta^{R_n}_{S_n}(e)$, is defined to be
\begin{itemize}
    \item $d^0_{R_n}(e)\overline{d^0_{S_n}(e)}$ if $e(1) = n$, or
    \item $\overline{d^1_{S_n}(e)}d^1_{R_n}(e)$ if $e(0) = n$.
\end{itemize}

Note that $\Delta^{R_n}_{R_n}(e) = d_{R_n}(e)$.
\end{definition}

As before, a \emph{deformation} $h$ of an equipment diagram $(\varphi : G \to R, \Ra)$ into $(\phi : G \to R, \Sa)$ is a continuous function $h : G \times [0,1] \to R$ for which $h(-,0) = \varphi$, $h(-,1) = \phi$, and $h(-,t)$ is an equipment diagram for all $t$. We do not ask that the choice of node-rectangles vary continuously in $t$. This means that the constant function at $\varphi$ may be considered a deformation from $(\varphi, \Ra)$ to $(\varphi, \Sa)$ for any suitable $\Ra$ and $\Sa$. 

\begin{definition}
Let $n$ be a node in $G$, let $(\varphi : G \to R, \Ra)$ and $(\phi, \Sa)$ be equipment diagrams, and let $h$ be a deformation from $\varphi$ to $\phi$. Given a valuation $\mu$ of $(\varphi, \Ra)$, the \emph{induced valuation} $\mu^{\Ra}_{\Sa}$ of $(\phi, \Sa)$ is equal to $\mu$ on all regions and edges, but at each node $n$ is defined to be $\mu(n)$ composed with $\big(\Delta^{R_n}_{S_n}(e)\big)\inv$ for all vertical edges $e$ incident to $n$.
\end{definition}

That the composite involved in defining $\mu^{\Ra}_{\Sa}(n)$ exists can be shown by a routine combinatorial argument on $n$'s signature, and its uniqueness can be deduced from the general associativity of composition. Note that the constant function at $\varphi$ is always a deformation from $\varphi$ to itself, that a deformation may be reversed, and that deformations from $\varphi$ to $\phi$ and from $\phi$ to $\psi$ may be concatenated into a deformation from $\varphi$ to $\psi$.
\begin{lemma}{Functoriality}\label{Lem_InducedValueFacts}
Let $(\varphi, \Ra) \xto{h} (\phi, \Ta) \xto{h'} (\psi, \Sa)$ be a chain of equipment diagram deformations, and let $\mu$ be a valuation of $(\varphi, \Ra)$. Then
\begin{enumerate}
    \item $\mu^{\Ra}_{\Ra} = \mu$,
    \item $\left( \mu^{\Ra}_{\Ta} \right)^{\Ta}_{\Sa} = \mu^{\Ra}_{\Sa}$, and 
    \item $\left( \mu^{\Ra}_{\Ta}\right)^{\Ta}_{\Ra} = \mu$.
\end{enumerate}
\end{lemma}
\begin{proof}
Equation (3) follows immediately from (1) and (2). Equation (1) holds because $\mu^{\Ra}_{\Ra}(n)$ is $\mu(n)$ composed with $\big(\Delta^{R_n}_{R_n}(e)\big)\inv = d_{R_n}(e)\inv$ for vertical edges $e$ incident to $n$; by construction $d_{R_n}(e)$ corresponds to an identity 2-cell, and so does not change the value of $\mu(n)$. Now we turn to Equation (3).

Recall that $\mu^{\Ra}_{\Ta}(n)$ is $\mu(n)$ composed with $\big(\Delta^{R_n}_{T_n}(e)\big)\inv$, and $\left( \mu^{\Ra}_{\Ta} \right)^{\Ta}_{\Sa}(n)$ is $\mu^{\Ra}_{\Ta}(n)$ composed with $\big(\Delta^{T_n}_{S_n}(e)\big)\inv$ for the vertical edges $e$ incident to $n$. Consider, without loss of generality, the edges $e$ for which $e(1) = n$. Then $\big(\Delta^{R_n}_{T_n}(e)\big)\inv = d^0_{T_n}(e)\overline{d^0_{R_n}(e)}$ and $\big(\Delta^{T_n}_{S_n}(e)\big)\inv = d^0_{S_n}(e)\overline{d^0_{T_n}(e)}$, so their composite in $\left( \mu^{\Ra}_{\Ta} \right)^{\Ta}_{\Sa}(n)$ will be $$d^0_{S_n}(e)\overline{d^0_{T_n}(e)}d^0_{T_n}(e)\overline{d^0_{R_n}(e)} = d^0_{S_n}(e)\overline{d^0_{R_n}(e)} = \big(\Delta^{R_n}_{S_n}(e)\big)\inv.$$
But $\big(\Delta^{R_n}_{S_n}(e)\big)\inv$ is precisely what $\mu(n)$ is composed with to get $\mu^{\Ra}_{\Sa}(n)$. So, $\left( \mu^{\Ra}_{\Ta} \right)^{\Ta}_{\Sa} = \mu^{\Ra}_{\Sa}$.
\end{proof}

We seek to show that the value of equipment diagrams is invariant under deformation in the sense that $\mu(\varphi, \Ra) = \mu^{\Ra}_{\Sa}(\phi, \Sa)$ when there is a deformation of $\varphi$ into $\phi$.

\subsection{Tiling and Invariance}

In this section, we will define admissible tilings for equipment diagrams, and show that all admissible tilings of an equipment diagram have the same value. We proceed in much the same way as in Section \ref{SecTilings}.

An admissible tiling $T$ of an equipment diagram $(\varphi, R_n)$ is tiling for which $\varphi$ restricts to an equipment diagram $\varphi_i$ on each tile $T_i$, and for which each tile contains at most one node and connected component of $\varphi$. We further require that if $T_i$ contains a node $n$, then the node-rectangle around $n$ in $\varphi_i$ is $T_i$ itself. If $\Ta \subseteq T$ denotes the set of tiles which contain nodes, then this condition ensures that $(\varphi, \Ta)$ is also an equipment diagram.

\begin{construction}
There exists an admissible tiling $T_\varphi$ for any equipment diagram $(\varphi, \Ra)$.
\end{construction}
\begin{proof}
The tiling may be constructed as in Section \ref{Sec_DiagramsTileable}, but with even less care taken for vertical edges. Namely, we construct $T_\varphi$ by adding in the rectangles $R_n \in \Ra$, together with the tiling of each horizontal edge according to the proscription in Section \ref{Sec_DiagramsTileable}. We will then tile the vertical edges in a similar way, add those to $T_\varphi$, and then tile the rest naively and add those tiles as well.

To tile a vertical edge $e$ on $R - \bigcup \interior R_n$, first surround $e$ with a tube of diameter $\epsilon > 0$ so that it avoids all other tubes around all other edges (vertical or horizontal). Then, moving along the orientation of $e$, we will place a rectanglular tile within this $\epsilon$-tube. Restricted to one such tile, $\varphi$ clearly satisfies conditions (1) of Def. \ref{Def_EquipmentDiagram} (since there are no horizontal edges), (3) (there are no nodes), and (2a) and (2b) (since $\varphi$ satisfies these). The only nontrivial condition to satisfy is (2c), which, since there are no nodes, means that we must place tiles so that $e(0)$ and $e(1)$ are on opposite vertical sides or $e(0) <_v e(1)$ when restricted to a tile.

Note that by condition (2b) of $\varphi$, if $e$ intersects enters and exits one of our tiles through the same vertical boundary, then we must have $e(0) <_v e(1)$ restricted to that tile. Therefore, the only constraint we must satisfy is that $e$ can never leave through the top side of a tile, or equivalently, can never enter through the bottom. So we will tile $e$ naively, and then adjust our tiling to match this constraint.

Tile $e$ naively by placing a rectangle in the $\epsilon$-tube around $e$ that does not overlap with any other rectangles already placed and only contains a single connected component of $e$. These tiles are ordered by the ordering on $e$. Note that this tiling begins at a node-rectangle and ends at one as well, so $e$ cannot enter the first tile through the bottom (it would have to leave the node-rectangle through the top) or leave the last tile through the top. Therefore, the problematic tiles will be bookended by pairs of tiles, one which $e$ enters through a valid side but leaves through the top, and one which $e$ enters from the bottom but leaves through a valid side. Let $T_i$ be a tile which $e$ leaves through the top, but enters through a side that is not the bottom. Let $T_j$ be the next tile in the tilings which $e$ exits through side which is not the top. Therefore, $e$ must enter $T_k$ through the bottom and leave through the top for all $i < k < j$. By (2b) and continuity, if $e(t_0)$ is in $T_k$ then for all $t > t_0$, $e(t) >_h e(t_0)$ or $e(t) <_h e(t_0)$ in $T_k$. Therefore, $e$ may be tiled within $T_k$ by a number of tiles, through each of which $e$ flows left to right or right to left, except for the first, which $e$ enters through the bottom, and last, which $e$ leaves through the top. 
$$\ingraph{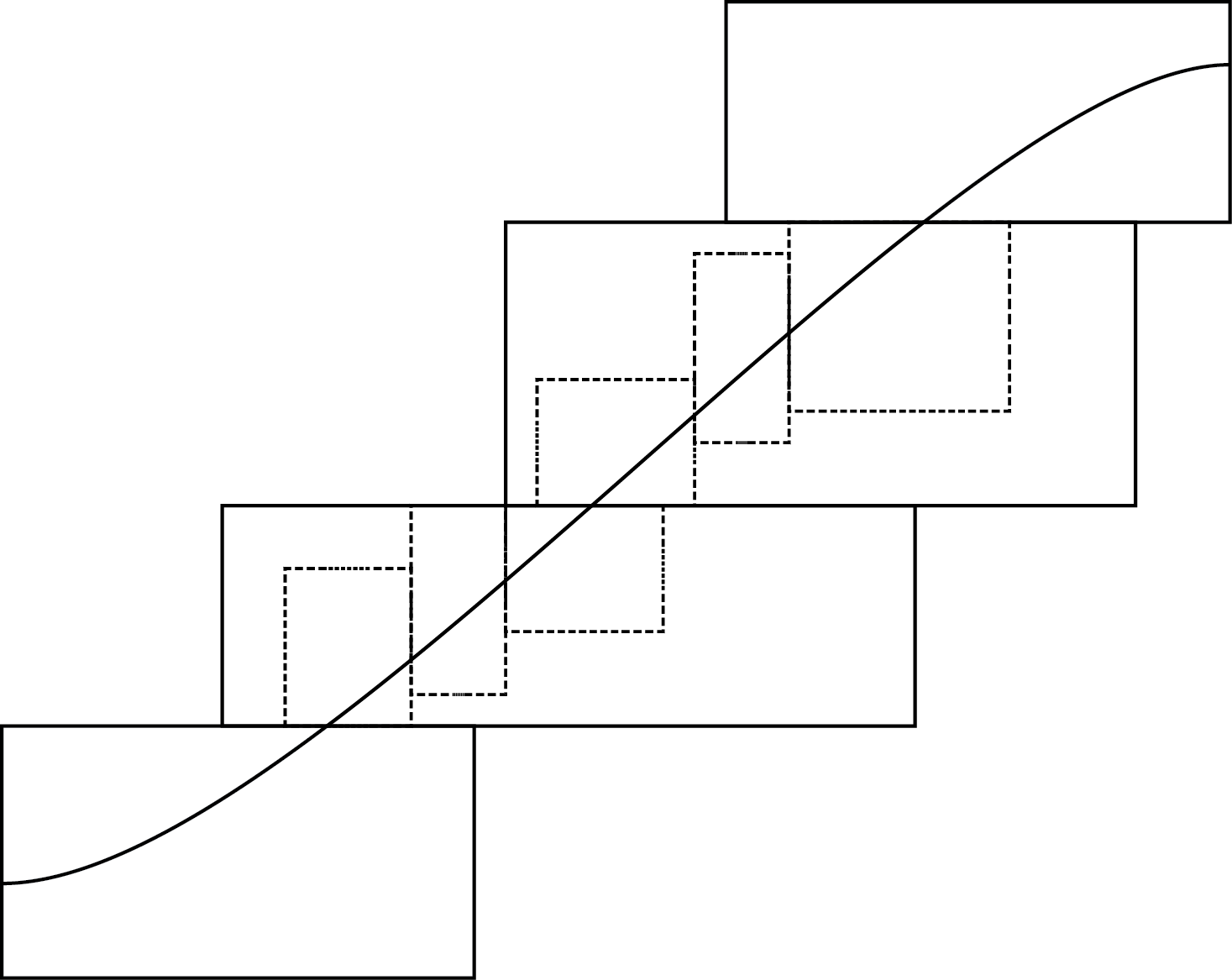}$$

Having done this to each $T_k$, we may restrict ourselves to problem-pairs ($T_i$, $T_j$) where $j = i + 1$. Let $T'$ be the rectangle formed by extruding the intersection of the top side of $T_i$ with the bottom side of $T_j$ downwards to the bottom of $T_i$ and upwards to the top of $T_j$. Since $e$ enters $T_i$ through a valid side, and since by construction $T_j$ only contains one connected component of $e$, $e$ cannot enter $T'$ through the bottom or top, and therefore must enter through the left or right, which is valid. Likewise, since $e$ leaves $T_j$ validly, and $T_i$ contains only one connected component of $e$, $e$ cannot leave $T'$ through the top or bottom, and so must left or right, which is valid. Replace $T_i$ by $T_i - T'$ and $T_j$ by $T_j - T'$, and add $T'$ into the tiling. 
$$\ingraph{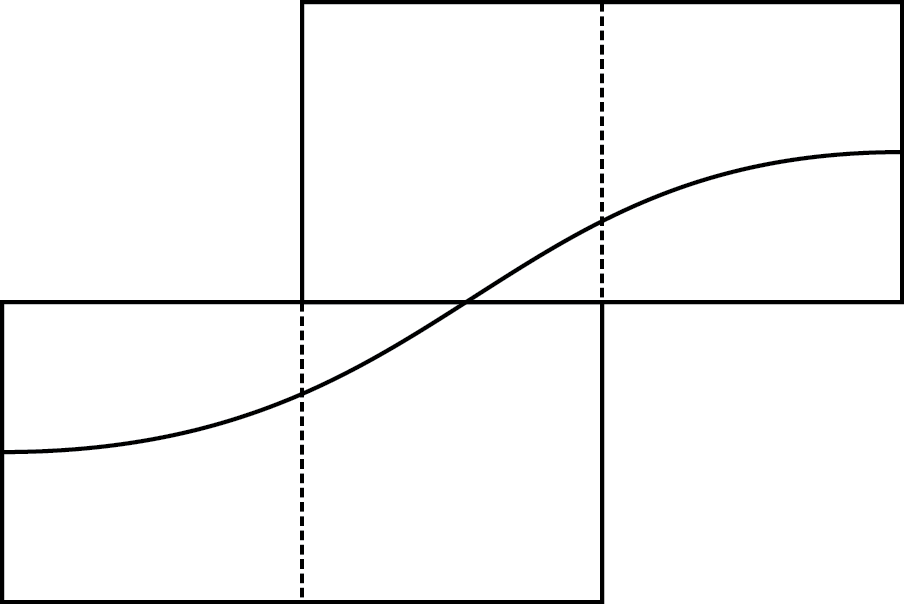} \quad \longrightarrow \quad \ingraph{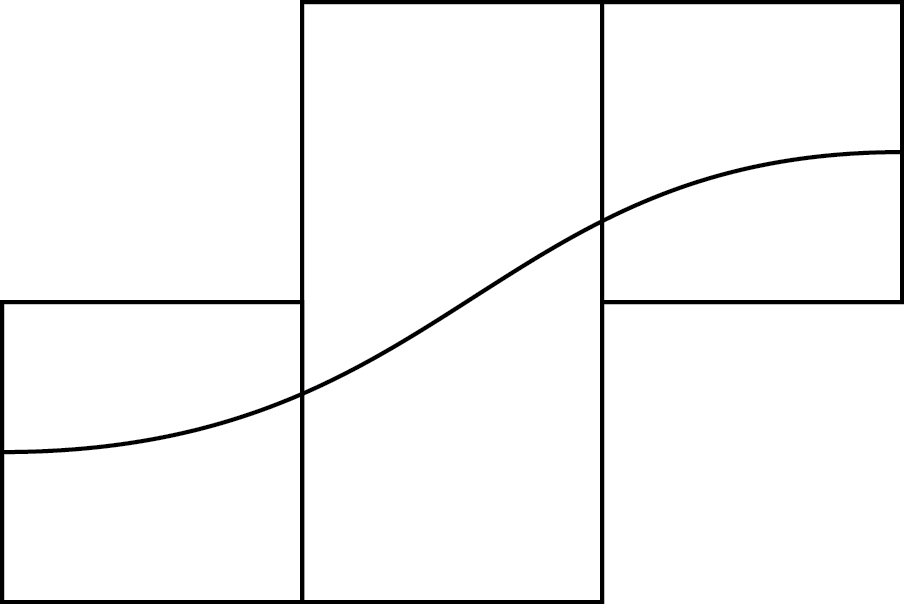}$$
Having done this for all problem pairs, we are left with an admissible tiling of $e$.
\end{proof}

If $\mu : \varphi \to \Ea$ is a valuation for $(\varphi, \Ra)$, and $T$ an admissible tiling, then $T$ denotes a compatible arrangement of $2$-cells in $\Ea$ as follows:
\begin{itemize}
    \item If the restriction $\varphi_i$ contains the node $n$, then $T_i$ denotes the double cell $\mu^{\Ra}_{\Ta}(n)$.
    \item If the restriction $\varphi_i$ has no node, then it either contains a single edge or no edge. If it contains no edge, then it denotes the identity on the value of the region contained within it. If it contains a horizontal edge, then it denotes the identity on its value. If it contains a vertical edge $e$ then it denotes (co)unit of the value of $e$ corresponding to $d_{\varphi|T_i}(e) \in \Ka$. 
\end{itemize}

Since every compatible arrangement of cells in an equipment is composable (pinwheels may always be avoided by bending), the arrangment denoted by $T$ has a composite which we will denote $\mu(T)$. Note that by definition, $\mu(T) =  \mu^{\Ra}_{\Ta}(T)$. As before, we would like to define $\mu(\varphi, \Ra)$ to be $\mu(T)$ for an admissible tiling $T$, but we need to show that $\mu(\varphi)$ is invariant under choices of such tilings. We will therefore prove an analogue of Proposition \ref{Lem_TileInvariant} in an analogous way.

\begin{proposition}\label{Lem_TileInvariantEq}
Let $\varphi$ be an equipment diagram and $\mu$ a valuation for $\varphi$. If $S$ and $T$ are admissible diagrams for $\varphi$, then $\mu(S) = \mu(T)$.
\end{proposition}

We will prove this using a few helper lemmas.

\begin{lemma}\label{Lem_InvariantRefinementEq}
If $S$ and $T$ are admissible tilings for an equipment diagram $(\varphi, \Ra)$, and $S$ is a refinement (see Def. \ref{Def_Refinement}) of $T$, then $\mu(S) = \mu(T)$ for any valuation $\mu$ of $\varphi$.
\end{lemma}
\begin{proof}
We will show that $\left( \mu^{\Ra}_{\Ta} \right)^{\Ta}_{\Sa}(S) =  \mu^{\Ra}_{\Ta}(T)$. Then, by Lemma \ref{Lem_InducedValueFacts}, 
$$\mu(S) = \mu^{\Ra}_{\Sa}(S) = \left( \mu^{\Ra}_{\Ta} \right)^{\Ta}_{\Sa}(S) = \mu^{\Ra}_{\Ta}(T) = \mu(T).$$

Let $T_i$ be a tile of $T$ and let $X \subseteq S$ be the set of tiles dividing it in $S$. We will show that $ \left( \mu^{\Ra}_{\Ta} \right)^{\Ta}_{\Sa}(X) = \mu^{\Ra}_{\Ta}(T_i)$, from which it will follow that $\left( \mu^{\Ra}_{\Ta} \right)^{\Ta}_{\Sa}(S) =  \mu^{\Ra}_{\Ta}(T)$.

Either $T_i$ contains a node or it does not. Suppose it does not. It if its empty, then the valuations $\left( \mu^{\Ra}_{\Ta} \right)^{\Ta}_{\Sa}$ and $\mu^{\Ra}_{\Ta}$ are both equal to $\mu$ on $T_i$. Since $T_i$ is empty, all the $x \in X$ are similarly empty, and they evaluate to the identity of the region contained in $T_i$. A composite of identities is an identity, so their composite value will equal the identity of the region contained in $T_i$, which is the value of $T_i$. 

If $T_i$ contains a horizontal edge $e$, then similarly the valuations $\left( \mu^{\Ra}_{\Ta} \right)^{\Ta}_{\Sa}$ and $\mu^{\Ra}_{\Ta}$ are both equal to $\mu$ on $T_i$. There is then some $x \in X$ which also intersects $e$ and $\mu(x)$ will be the identity on $\mu(e)$. Since there is only one edge in $T_i$, no nontrivial whiskering can take place, and all other tiles in $X$ evaluate either to the identity of a region or to the identity of $e$. Therefore, $\left( \mu^{\Ra}_{\Ta} \right)^{\Ta}_{\Sa}(X) = \mu(e) = \mu^{\Ra}_{\Ta}(T_i)$. 

Suppose $T_i$ contains a vertical edge $e$ so that $\mu^{\Ra}_{\Ta}(T_i)$ is the (co)unit on $\mu(e)$ given by $d_{T_i}(e)$. The tiles of $X$ that $e$ intersects as it winds through $T_i$ may be organized into a sequence $(\sigma_j)_{1 \leq j \leq k}$ according to the ordering on $e$. Note that $d^0_{\sigma_1}(e) = d^0_{T_i}(e)$ and $d^1_{\sigma_k}(e) = d^1_{T_i}(e)$. If $e$ leaves $\sigma_j$ on side $x$, it must enter $\sigma_{j+1}$ on side $\bar{x}$. For this reason, the directions $d_{\sigma_j}(e)$ form a composable sequence in $\Ka$, and so by Lemma \ref{Lem_WordsCompose}, their composite equals $$d^0_{\sigma_1}(e)d^1_{\sigma_k}(e) = d^0_{T_i}(e)d^1_{T_i}(e) = d_{T_i}(e).$$
Since all other tiles in $X$ are identities, the total composite $\left( \mu^{\Ra}_{\Ta} \right)^{\Ta}_{\Sa}(X)$ will equal the (co)unit on $\mu(e)$ given by $d_{T_i}(e)$, which is $\mu^{\Ra}_{\Ta}(T_i)$.

Suppose that $T_i$ contains a node $n$. Let $S_n \in X$ be the tile which also contains $n$. Horizontal edges and blank tiles in $X$ will contribute only identities which are the same in both valutations, so their contribution can safely be ignored. Suppose, without loss of generality, that $e$ is a vertical edge incident to $n$ with $e(1) = n$. Then $\left( \mu^{\Ra}_{\Ta} \right)^{\Ta}_{\Sa}(S_n) = \left( \mu^{\Ra}_{\Ta} \right)^{\Ta}_{\Sa}(n)$ involves the composite of $\mu^{\Ra}_{\Ta}(n)$ with $\big(\Delta^{T_n}_{S_n}(e)\big)\inv = d^0_{S_n}(e)\overline{d^0_{T_i}(e)}$. Consider the sequence $(\sigma_j)_{1 \leq j \leq k} \subseteq X - \{S_n\}$ though which $e$ winds as it makes its way from the boundary of $T_i$ to the boundary of $S_n$. By an argument similar to that above, the compsite of the values of the cells $(\sigma_j)$ will equal $d^0_{T_i}(e)\overline{d^0_{S_n}(e)}$. The full composite in $\left( \mu^{\Ra}_{\Ta} \right)^{\Ta}_{\Sa}(X)$ then includes the composite $$d^0_{T_i}(e)\overline{d^0_{S_n}(e)}d^0_{S_n}(e)\overline{d^0_{T_i}(e)} = d^0_{T_i}(e)\overline{d^0_{T_i}(e)} = d_{T_i}(e).$$
The latter precisely the value of $e$ in $\mu^{\Ra}_{\Ta}(T_i)$, so in total, $\left( \mu^{\Ra}_{\Ta} \right)^{\Ta}_{\Sa}(X) = \mu^{\Ra}_{\Ta}(T_i)$.
\end{proof}

\begin{lemma} \label{Lem_RefinementAdmissableEq}
Let $S$ and $T$ be admissible tilings of an equipment diagram $(\varphi, \Ra)$. Then they admit a common refinement which is also admissible. 
\end{lemma}
\begin{proof}
As in Lemma \ref{Lem_RefinementAdmissable}, we will begin by showing that $\varphi$ restricts to an equipment diagram on each tile $S_i \cap T_j$ of the largest common refinement $S \# T$ of $S$ and $T$. Then we can simply re-tile those tiles which have more than one connected component.

We will show that $\varphi$ restricts to an equipment diagram on $R_{ij} = S_i \cap T_j$. We verify the 3 conditions of Def. \ref{Def_EquipmentDiagram}.

\begin{enumerate}
    \item The horizontal restriction $\varphi_h$ is a horizontal diagram because it is for $S_i$ and $T_j$.
    \item Condition (a) is satisfied because of the definition of the boundary and because it is satisfied for $S_i$ and $T_j$. Condition (b) is satisfied because it is satisfied for $S_i$ and $T_j$. Condition (c) is satisfied for all edges which intersect only the boundaries of $S_i$ or $T_j$, or a node, since $S_i$ and $T_j$ satisfy this condition. If $e$ has, without loss of generality, $e(0)$ on a side of $S_i$ and $e(1)$ on a side of $T_j$ which are distinct in $R_{ij}$, $e(0)$ must be on one of the top $3$ sides of $S_i$ and $e(1)$ on the bottom $3$ sides of $T_j$. If $e(0)$ is on the top of $S_i$ then it is on the top of $R_{ij}$, and the condition is satisfied; similarly, if $e(1)$ is on the bottom of $T_j$ then it is on the bottom of $R_{ij}$ and the condition is satisfied. If neither of these cases hold, then they are on opposite vertical boundaries of $R_{ij}$ and the condition does not apply. 
    \item Since the node-rectangle of a node in $R_{ij}$ (if there is one) is $R_{ij}$ itself, the conditions are satisfied by the discussion above. For good measure, note that since $R_{ij} \subseteq S_i$, if $R_{ij}$ contains a node $n$ then $S_i$ does as well, and therefore all edges which intersect $R_{ij}$ also intersect $S_i$ and are then incident to $n$. 
\end{enumerate}

\end{proof}

We can now prove Proposition \ref{Lem_TileInvariantEq}.

\begin{proof}[Proof of Proposition \ref{Lem_TileInvariantEq}]
By Lemma \ref{Lem_RefinementAdmissableEq}, there is an admissible common refinement $X$ of $S$ and $T$. By Lemma \ref{Lem_InvariantRefinementEq}, $\mu(S) = \mu(X)$ and $\mu(T) = \mu(X)$. So $\mu(S) = \mu(T)$.
\end{proof}

Finally, we can prove that the value of an equipment diagram is invariant under deformation.
\begin{theorem}[Invariance Under Deformation] \label{Thm:InvariantValueEq}
Let $(\varphi, \Ra) \xto{h} (\phi, \Sa)$ be a deformation of equipment diagrams, and let $\mu$ be a valuation of $(\varphi, \Ra)$. Then
$$\mu(\varphi, \Ra) = \mu^{\Ra}_{\Sa}(\phi, \Sa).$$
\end{theorem}
\emph{Proof.} By compactness, we may choose $(t_i)_{1 \leq i \leq k} \subseteq [0,1]$ such that $t_1 = 0$, $t_k = 1$, and $T_{h(-,t_i)}$ is an admissible tiling for $h(-, t_{i + 1})$. Let $\varphi_i = h(-,t_i)$ and let $\Ra_i$ denote the choice of node-rectangles for $\varphi_i$. Note that $\mu^{\Ra}_{\Ra_i}$ is the valuation induced by $\mu$ on $\varphi_i$, and that $\mu^{\Ra}_{\Ra_i}(\varphi_i, \Ra_i) = \mu^{\Ra}_{\Ra_i}(T_{\varphi_i})$ by definition. Since $T_{\varphi_i}$ is an admissible tiling for $\varphi_{i+1}$, $\mu^{\Ra}_{\Ra_{i+1}}(\varphi_{i+1}, \Ra_{i+1}) = \mu^{\Ra}_{\Ra_{i+1}}(T_i)$ by Proposition \ref{Lem_TileInvariantEq}. Then
\begin{align*}
    \mu^{\Ra}_{\Ra_i}(\varphi_i, \Ra_i) &= \mu^{\Ra}_{\Ra_i}(T_{\varphi_i}) \\
        &= \left(\mu^{\Ra}_{\Ra_i}\right)^{\Ra_i}_{\Ta_{\varphi_i}}(T_{\varphi_i}) &\mbox{by the definition of $\mu^{\Ra}_{\Ra_i}(T_{\varphi_i})$} \\
        &= \mu^{\Ra}_{\Ta_{\varphi_i}}(T_{\varphi_i}) &\mbox{by Lemma \ref{Lem_InducedValueFacts}}\\
        &= \left(\mu^{\Ra}_{\Ra_{i+1}}\right)^{\Ra_{i+1}}_{\Ta_{\varphi_i}}(T_{\varphi_i})  &\mbox{by Lemma \ref{Lem_InducedValueFacts}}\\
        &= \mu^{\Ra}_{\Ra_{i+1}}(T_{\varphi_i}) &\mbox{by the definition of $\mu^{\Ra}_{\Ra_{i+1}}(T_{\varphi_i})$}\\
        &= \mu^{\Ra}_{\Ra_{i+1}}(\varphi_{i+1}, \Ra_{i+1}) &\mbox{for all $i$.}
\end{align*}

Therefore, $$ \mu(\varphi, \Ra) = \mu^{\Ra}_{\Ra_1}(\varphi_1, \Ra_1) = \mu^{\Ra}_{\Ra_k}(\varphi_k, \Ra_k) = \mu^{\Ra}_{\Sa}(\phi, \Sa). \qed$$

\bibliography{Bibliography}

\end{document}